\documentclass{article}
\usepackage[T1]{fontenc}
\usepackage[utf8]{inputenc}
\usepackage[style=alphabetic,sorting=nyt,giveninits=true]{biblatex}
\usepackage{geometry}
\usepackage{parskip}
\usepackage{amsmath}
\usepackage{amssymb}
\usepackage{amsthm}
\usepackage{hyperref}
\usepackage[capitalize,nameinlink,noabbrev]{cleveref}
\usepackage{mathtools}
\usepackage{commath}
\usepackage{bm}
\usepackage{tikz}
\usetikzlibrary{cd,fit,shapes.geometric,backgrounds}
\usepackage{adjustbox}
\usepackage{array}
\usepackage{enumitem}
\usepackage{mathrsfs}
\usepackage{tcolorbox}
\tcbuselibrary{skins}
\usepackage{xparse}
\usepackage{microtype}
\usepackage{listings}
\usepackage{longtable}

\makeatletter
\if@cref@capitalise
\crefname{subsection}{Subsection}{Subsections}
\else
\crefname{subsection}{subsection}{subsections}
\fi
\makeatother
\Crefname{subsection}{Subsection}{Subsections}

\makeatletter
\newcount\my@repeat@count
\newcommand{\myrepeat}[2]{%
  \begingroup
  \my@repeat@count=\z@
  \@whilenum\my@repeat@count<#1\do{#2\advance\my@repeat@count\@ne}%
  \endgroup
}
\makeatother

\newcommand{\indentrule}[1][1]{\myrepeat{#1}{\hspace{5pt}\vrule\hspace{10pt}}}

\newcommand*{\citeMacaulay}{\citelink{M2}{\code{Macaulay2}}}

\newenvironment{customProof}[1][Proof]{\let\oldProofName\proofname\renewcommand*{\proofname}{#1}\begin{proof}}{\end{proof}\renewcommand*{\proofname}{\oldProofName}}

\makeatletter
\g@addto@macro\bfseries{\boldmath}
\makeatother

\makeatletter
\def\thmCite{\@ifnextchar[{\@with}{\@without}}
\def\@with[#1]#2{{\normalfont\cite[#1]{#2}\;}}
\def\@without#1{{\normalfont\cite{#1}\;}}
\makeatother

\ExplSyntaxOn
\DeclareExpandableDocumentCommand{\IfNoValueOrEmptyTF}{mmm}
 {
  \IfNoValueTF{#1}{#2}
   {
    \tl_if_empty:nTF {#1} {#2} {#3}
   }
 }
\ExplSyntaxOff

\newcommand*{\valOrBlank}[1]{\IfNoValueOrEmptyTF{#1}{}{#1}}

\NewDocumentCommand{\mySum}{e{_^}}{%
	\sum_{\mathclap{\valOrBlank{#1}}}^{\mathclap{\valOrBlank{#2}}}\hspace{0.025cm}%
}

\NewDocumentCommand{\myCap}{e{_^}}{%
	\bigcap_{\mathclap{\valOrBlank{#1}}}^{\mathclap{\valOrBlank{#2}}}\hspace{0.05cm}%
}

\tcbset{sharp corners, colback=white, colframe=black, boxrule=0.4pt,parbox=false}

\newcommand{\tcbrule}{%
\begin{tcolorbox}[
    sharp corners,
    boxrule=0mm,
    enhanced,
	frame hidden,
	height=0.4pt,
	colback=black
]
\end{tcolorbox}
}

\newcommand{\tcbdoublerule}{%
\begin{tcolorbox}[
    sharp corners,
    boxrule=0mm,
    enhanced,
    borderline north={0.4pt}{0pt}{black},
    borderline south={0.4pt}{0pt}{black},
	frame hidden,
	height=2pt
]
\end{tcolorbox}
}

\makeatletter
\newcommand*{\defeq}{\hspace{-0.04cm}\mathrel{\rlap{%
                     \raisebox{0.4ex}{$\scriptstyle\m@th\cdot$}}%
                     \raisebox{-0.05ex}{$\scriptstyle\m@th\cdot$}}%
				 =}
\makeatother

\tikzset{square/.style={regular polygon,regular polygon sides=4}}

\begingroup
\makeatletter
\@for\theoremstyle:=definition,remark,plain\do{%
	\expandafter\g@addto@macro\csname th@\theoremstyle\endcsname{%
		\addtolength\thm@preskip\parskip
	}%
}
\endgroup

\let\oldproof\proof
\def\proof{\oldproof\unskip}

\makeatletter
\def\footnoterule{
  \hrule \@width 9cm \kern 3\p@} 
\makeatother

\DeclareMathOperator{\Ann}{Ann}
\DeclareMathOperator{\Ass}{Ass}

\DeclareMathOperator{\dirlim}{\mathop{\lim_{\longrightarrow}}}
\DeclareMathOperator{\hgt}{height}
\DeclareMathOperator{\id}{Id}
\DeclareMathOperator{\grade}{grade}
\DeclareMathOperator{\Spec}{Spec}
\DeclareMathOperator{\Supp}{Supp}
\DeclareMathOperator{\LF}{\mathsf{LF}}

\newcommand*{\modcat}[1]{\ensuremath{\textstyle\mathbf{#1\text{\normalfont\bfseries -Mod}}}}

\DeclarePairedDelimiter{\smallAbs}{|}{|}

\newtheorem{theorem}{Theorem}[section]
\newtheorem{lemma}[theorem]{Lemma}
\newtheorem{corollary}[theorem]{Corollary}
\newtheorem{proposition}[theorem]{Proposition}
\newtheorem{definition}[theorem]{Definition}
\newtheorem{conjecture}[theorem]{Conjecture}
\newtheorem{notation}[theorem]{Notation}
\theoremstyle{definition}
\newtheorem{algorithm}[theorem]{Algorithm}
\newtheorem{example}[theorem]{Example}
\newtheorem{remark}[theorem]{Remark}
\theoremstyle{remark}
\newtheorem*{note}{Note}

\newcommand{\blockComment}[1]{}

\newcommand{\code}[1]{{\normalfont\ttfamily #1}}

\DeclareTextFontCommand{\emph}{\boldmath\bfseries}

\newcommand{\xra}[1]{\xrightarrow{\,\,#1\,\,}}
\newcommand{\xhra}[1]{\xhookrightarrow{\,\,#1\,\,}}

\newcommand{\citelink}[2]{{\hypersetup{linkbordercolor={0 1 0}}\hyperlink{cite.\therefsection @#1}{#2}}}

\AtBeginBibliography{\vspace*{8pt}}

\DeclareFieldFormat[article,book,incollection,misc]{title}{\mkbibitalic{#1}}
\DeclareFieldFormat[article]{journaltitle}{#1\isdot}
\DeclareFieldFormat[article,incollection]{booktitle}{\normalfont{#1}}

\title{\textsf{LF}-Covers and Binomial Edge Ideals of K\"{o}nig Type}
\author{David Williams\footnote{The author was supported by EPSRC grant EP/T517835/1}}
\date{\vspace{-0.3cm}}

\makeatletter
\renewcommand*{\@fnsymbol}[1]{\@arabic{#1}}
\makeatother

\begin{document}

\maketitle

\begin{abstract}
	\noindent We give a combinatorial characterisation of connected graphs whose binomial edge ideals are of K\"{o}nig type, developed independently to the similar characterisation given in \cite[Lemma 5.3]{laclairInvariantsBinomialEdge2023}, and exhibit some classes of graphs satisfying our criteria.

	\vspace{0.35\baselineskip}
	\noindent For any connected Hamiltonian graph $G$ on $n$ vertices, we compute an explicit root of $H_{\mathcal{J}(G)}^{n-1}(R)$ as an $F$-finite $F$-module.
\end{abstract}

\begin{tcolorbox}
	Unless specified otherwise, let
	\begin{equation*}
		R=k[x_1,\ldots,x_n,y_1,\ldots,y_n]
	\end{equation*}
	for some field $k$ (of arbitrary characteristic) and $n\geq 2$. Furthermore, set $\delta_{i,j}\defeq x_iy_j-x_jy_i$ for $1\leq i<j\leq n$.

	Throughout this paper, all graphs will be finite, simple and undirected. For a graph $G$, we denote by $V(G)$ and $E(G)$ the sets of vertices and edges $\{i,j\}$ of $G$ respectively. Unless specified otherwise, all graphs will have at most $n$ vertices, with vertex set $\{1,\ldots,\smallAbs{V(G)}\}$.

	We denote by $P_m$, $C_m$ and $K_m$ the path, cycle and complete graphs on $m$ vertices respectively, with $K_0$ denoting the null graph with $V(K_0)=\varnothing$ and $E(K_0)=\varnothing$.
\end{tcolorbox}

\section{Preliminaries}

\subsection{Binomial Edge Ideals}

Our objects of study are the binomial edge ideals of graphs. These  were introduced in \cite{herzogBinomialEdgeIdeals2010}, and independently in \cite{ohtaniGraphsIdealsGenerated2011}.

We begin with their definition:

\begin{definition}\label{BEIDef}
	Let $G$ be a graph. We define the \emph{binomial edge ideal} of $G$ as
	\begin{equation*}
		\mathcal{J}(G)\defeq(\delta_{i,j}:\{i,j\}\in E(G))
	\end{equation*}
	We may also denote this ideal by $\mathcal{J}_G$.
\end{definition}

There is a rich interplay between the algebraic properties of these ideals and the combinatorial properties of the corresponding graph. For an overview of some of these interactions, see \cite[Chapter 7]{herzogBinomialIdeals2018}.

A key fact about binomial edge ideals is the following:

\begin{theorem}\thmCite[Corollary 2.2]{herzogBinomialEdgeIdeals2010}\label{BEIRadical}
	Binomial edge ideals are radical.
\end{theorem}

In fact, an explicit description of the primary decomposition can be given by purely combinatorial means, but we must first introduce some notation:

\begin{notation}
	Let $G$ be a graph, and $S\subseteq V(G)$. Denote by $c_G(S)$ the number of connected components of $G\setminus S$ (we will just write $c(S)$ when there is no confusion). Then we set
	\begin{equation*}
		\mathcal{C}(G)\defeq\{S\subseteq V(G):\text{\normalfont$S=\varnothing$ or $c(S\setminus\{v\})<c(S)$ for all $v\in S$}\}
	\end{equation*}
	That is, when we ``add back'' any vertex in $S\in\mathcal{C}(G)$, it must reconnect some components of $G\setminus S$.
\end{notation}

\begin{notation}\label{PSGDef}
	Let $G$ be a graph, and $S\subseteq V(G)$. Denote by $G_1,\ldots,G_{c(S)}$ the connected components of $G\setminus S$, and let $\tilde{G}_i$ be the complete graph with vertex set $V(G_i)$. Then we set
	\begin{equation*}
		P_S(G)\defeq(x_i,y_i:i\in S)+\mathcal{J}(\tilde{G}_1)+\cdots+\mathcal{J}(\tilde{G}_{c(S)})
	\end{equation*}
\end{notation}

\begin{proposition}
	$\mathcal{J}(K_S)$ is prime for any $S\subseteq\{1,\ldots,n\}$, and therefore $P_T(G)$ is prime for any graph $G$ and $T\subseteq V(G)$ also.
\end{proposition}

\begin{proof}
	This follows from \cite[Theorem 2.10]{brunsDeterminantalRings1988}.
\end{proof}

\begin{theorem}\thmCite[Corollary 3.9]{herzogBinomialEdgeIdeals2010}\label{BEIPD}
	Let $G$ be a graph. Then
	\begin{equation*}
		\mathcal{J}(G)=\myCap_{S\in\mathcal{C}(G)}P_S(G)
	\end{equation*}
	is the primary decomposition of $\mathcal{J}(G)$.
\end{theorem}

We will also make use of the following:

\begin{lemma}\thmCite[Lemma 3.1]{herzogBinomialEdgeIdeals2010}\label{HeightLemma}
	Let $G$ be a graph, and $S\in\mathcal{C}(G)$. Then
	\begin{equation*}
		\hgt_R(P_S(G))=\smallAbs{S}+n-c(S)
	\end{equation*}
\end{lemma}

\begin{remark}\label{DimGradeRemark}
	\cref{HeightLemma} tells us that
	\begin{equation*}
		\hgt_R(\mathcal{J}(G))=\min\{\smallAbs{S}+n-c(S):S\in\mathcal{C}(G)\}
	\end{equation*}
	In particular, the height of $\mathcal{J}(G)$ is independent of the characteristic of $k$.

	Furthermore, since $R$ is affine, we have
	\begin{equation*}
		\dim(R/\mathcal{J}(G))=\dim(R)-\hgt_R(\mathcal{J}(G))=2n-\hgt_R(\mathcal{J}(G))
	\end{equation*}
	(see, for example, \cite[Corollary 13.4]{eisenbudCommutativeAlgebraView1995}), and since $R$ is Cohen-Macaulay we have
	\begin{equation*}
		\grade_R(\mathcal{J}(G))=\hgt_R(\mathcal{J}(G))
	\end{equation*}
	(see, for example, \cite[Corollary 2.1.4]{brunsCohenMacaulayRings2005}), so the dimension and grade of $\mathcal{J}(G)$ are also independent of the characteristic of $k$.
\end{remark}

\subsection{Prime Characteristic Tools}

\begin{tcolorbox}
	Throughout this subsection, $R$ denotes any regular ring of prime characteristic $p>0$.
\end{tcolorbox}

We first define the Frobenius Functor introduced in \cite{peskineDimensionProjectiveFinie1973}, which exploits the additional structure afforded to rings of prime characteristic by the Frobenius endomorphism:

\begin{definition}
	We define a ring $F_*R$ with elements $\{F_*r:r\in R\}$ which inherits the same addition and multiplication as $R$. We turn this into an $R$-module by setting $r\cdot F_*s=F_*(r^ps)$. Then the \emph{Frobenius functor} $F_R:\modcat{R}\to\modcat{F_*R}$ sends $M\mapsto F_*R\otimes_R M$. We can then identify $F_*R$ with $R$ to view this as a functor $F_R:\modcat{R}\to\modcat{R}$. We denote by $F_R^e$ the $e$th iterated application of $F_R$.
\end{definition}

This is a somewhat unusual construction, which leaves us with
\begin{equation*}
	 r\cdot(s\otimes_R m)=(rs)\otimes_R m
\end{equation*}
for any $r\in R$ and $s\otimes_R m\in F_R(M)$, and
\begin{equation*}
	r\otimes_R(sm)=(s^pr)\otimes_R m
\end{equation*}
for any $r,s\in R$ and $m\in M$.

\begin{note}
	We may identify $F_R(R)$ with $R$ via $r\otimes_R s\mapsto s^pr$. More generally, we have $F_R(R/I)\cong R/I^{[p]}$, where $I^{[p]}$ denotes the ideal $\{a^p:a\in I\}$  This follows from the calculation that, for a map $A:R^n\to R^m$ given by an $m\times n$ matrix $(a_{i,j})$, the map $F_R(A):R^n\to R^m$ is given by the $m\times n$ matrix $A^{[p]}=(a_{i,j}^p)$.
\end{note}

The following result is crucial in exploiting many of the properties of this functor:

\begin{theorem}[Kunz's Theorem]\thmCite[Corollary 2.7]{kunzCharacterizationsRegularLocal1969}\label{Kunz}
	A (commutative) Noetherian ring $R$ is regular if and only if it is reduced and $F_R:\modcat{R}\to\modcat{R}$ is exact.
\end{theorem}

We now define $F$-finite $F$-modules, which were introduced in \cite{lyubeznikFmodulesApplicationsLocal1997}:

\begin{definition}
	We say that an $R$-module $\mathscr{M}$ is an \emph{$F$-finite $F$-module} if it is of the form
	\begin{equation*}
		\mathscr{M}=\dirlim\left[
		\begin{tikzcd}[column sep=0.7cm]
			M \arrow[r,"\beta"] & F_R(M) \arrow[r,"F_R(\beta)"] &[0.5cm] F_R^2(M) \arrow[r,"F_R^2(\beta)"] &[0.5cm] \cdots
		\end{tikzcd}
		\right]
	\end{equation*}
	for some finitely generated $R$-module $M$ and $R$-homomorphism $\beta:M\to F_R(M)$. In this case, we say that $\beta$ is a \emph{generating morphism} of $\mathscr{M}$. If $\beta$ is also injective, say that it is a \emph{root morphism} of $\mathscr{M}$, and that $M$ is a \emph{root} of $\mathscr{M}$. Every $F$-finite $F$-module has a root (see \cite[Proposition 2.3(c)]{lyubeznikFmodulesApplicationsLocal1997}).
\end{definition}

The key example of an $F$-finite $F$-module for our purposes is the following:

\begin{example}\label{LCFFinFMod}
	Let $\mathfrak{a}=(r_1,\ldots,r_d)$ be an ideal of $R$, and set $u=(r_1\cdots r_d)^{p-1}$. Then we have
	\begin{align*}
		H_\mathfrak{a}^d(R)&\cong\dirlim\left[R/\mathfrak{a}\xra{r_1\cdots r_d\cdot}R/\mathfrak{a}^{[2]}\xra{r_1\cdots r_d\cdot}R/\mathfrak{a}^{[3]}\xra{r_1\cdots r_d\cdot}\cdots\right]\\
		&\cong\dirlim\left[R/\mathfrak{a}\xra{u\cdot}R/\mathfrak{a}^{[p]}\xra{u^p\cdot}R/\mathfrak{a}^{[p^2]}\xra{u^{p^2}\cdot}\cdots\right]\\
		&\cong\dirlim\left[R/\mathfrak{a}\xra{u\cdot}F_R(R/\mathfrak{a})\xra{F_R(u\cdot)}F_R^2(R/\mathfrak{a})\xra{F_R^2(u\cdot)}\cdots\right]
	\end{align*}
	with the first isomorphism following from \cite[Exercise 5.3.7]{brodmannLocalCohomologyAlgebraic2013}.
\end{example}

Many properties of $F$-finite $F$-modules coincide with those of their roots, however we will only need one such correspondence for our purposes: that their associated primes agree.

This fact is well-known, we include a proof for completeness. We will make use of a result of Huneke and Sharp, \cite[Corollary 1.6]{hunekeBassNumbersLocal1993}, in doing so. This is stated in \cite{hunekeBassNumbersLocal1993} for regular local rings, but extends straightforwardly to the case of regular rings:

\begin{lemma}\label{FrobAssEqual}
	For any $R$-module $M$, we have $\Ass(M)=\Ass(F_R(M))$.
\end{lemma}

\begin{proof}
	Take any $\mathfrak{p}\in\Spec(R)$. Since $F_R$ commutes with localisation (see, for example, \cite[Remarks 1.0 (h)]{lyubeznikFmodulesApplicationsLocal1997}), we have $F_R(M)_\mathfrak{p}\cong F_{R_\mathfrak{p}}\hspace{-0.05cm}(M_\mathfrak{p})$ as $R_\mathfrak{p}$-modules.

	Furthermore, by \cite[Corollary 1.6]{hunekeBassNumbersLocal1993}, we have
	\begin{equation*}
		\Ass_{R_\mathfrak{p}}\hspace{-0.05cm}(F_{R_\mathfrak{p}}\hspace{-0.05cm}(M_\mathfrak{p}))=\Ass_{R_\mathfrak{p}}\hspace{-0.05cm}(M_\mathfrak{p})
	\end{equation*}
	Then
	\begin{align*}
		\mathfrak{p}\in\Ass_R(F_R(M))&\Longleftrightarrow\mathfrak{p}R_\mathfrak{p}\in\Ass_{R_\mathfrak{p}}\hspace{-0.05cm}(F_R(M)_\mathfrak{p})\\
		&\Longleftrightarrow\mathfrak{p}R_\mathfrak{p}\in\Ass_{R_\mathfrak{p}}\hspace{-0.05cm}(F_{R_\mathfrak{p}}\hspace{-0.05cm}(M_\mathfrak{p}))\\
		&\Longleftrightarrow\mathfrak{p}R_\mathfrak{p}\in\Ass_{R_\mathfrak{p}}\hspace{-0.05cm}(M_\mathfrak{p})\\
		&\Longleftrightarrow\mathfrak{p}\in\Ass_R(M)
	\end{align*}
	and we are done.
\end{proof}

\begin{proposition}\label{FirstRootAssInclusion}
	Suppose that $(R,\mathfrak{m})$ is local, and let $\mathscr{M}$ be an $F$-finite $F$-module with generating morphism $\beta:M\to F_R(M)$ for some finitely generated $R$-module $M$. Then
	\begin{equation*}
		\Ass_R(\mathscr{M})\subseteq\Ass_R(M)
	\end{equation*}
\end{proposition}

\begin{proof}
	Take any $\mathfrak{p}\in\Ass_R(\mathscr{M})$, so there exists some $\overline{x}\in\mathscr{M}$ such that $\Ann_R(\overline{x})=\mathfrak{p}$. Now, $\overline{x}$ corresponds to the image in $\mathscr{M}$ of some $x\in F_R^e(M)$ for some $e\geq 0$, and we may take $e$ to be minimal.

	Let
	\begin{equation*}
		\psi_i=F_R^{i-1}(\beta)\circ\cdots\circ F_R^e(\beta):F_R^e(M)\to F_R^i(M)
	\end{equation*}
	for each $i>e$, and set $\psi_e=\id_{F_R^e(M)}$. We then have
	\begin{equation*}
		\Ann_R(\overline{x})=\{r\in R:\psi_i(rx)=0\text{ for some }i\geq e\}
	\end{equation*}
	Since $R$ is Noetherian we can write $\mathfrak{p}=(a_1,\ldots,a_t)$ for some $a_i\in R$ and $t\geq 1$. Then, for each $1\leq i\leq t$, there exists some $l_i\geq e$ such that $\psi_{l_i}(a_ix)=0$.

	Now, let $l=\max\{l_i:1\leq i\leq t\}$. Then
	\begin{equation*}
		a_i\psi_l(x)=\psi_l(a_ix)=0
	\end{equation*}
	for each $1\leq i\leq t$, so setting $y=\psi_l(x)\in F_R^l(M)$ we have $\mathfrak{p}\subseteq\Ann_R(y)$.

	Conversely, if $r\in\Ann_R(y)$, then
	\begin{equation*}
		\psi_l(rx)=r\psi_l(x)=ry=0
	\end{equation*}
	and so $r\in\Ann_R(\overline{x})=\mathfrak{p}$. Then $\Ann_R(y)\subseteq\mathfrak{p}$, so $\Ann_R(y)=\mathfrak{p}$, and therefore $\mathfrak{p}\in\Ass_R(F_R^l(M))$.

	Then repeatedly applying \cref{FrobAssEqual} yields $\mathfrak{p}\in\Ass_R(M)$, and the result follows.
\end{proof}

\begin{lemma}\label{RootAssEqual}
	Suppose that $(R,\mathfrak{m})$ is local, and let $\mathscr{M}$ be an $F$-finite $F$-module with root morphism $\beta:M\hookrightarrow F_R(M)$ for some finitely generated $R$-module $M$. Then
	\begin{equation*}
		\Ass_R(\mathscr{M})=\Ass_R(M)
	\end{equation*}
\end{lemma}

\begin{proof}
	Since $M$ is a root of $\mathscr{M}$ the canonical inclusion is an injection, so $M\subseteq\mathscr{M}$. Then we have
	\begin{equation*}
		\Ass_R(M)\subseteq\Ass_R(\mathscr{M})
	\end{equation*}
	and so we are done by \cref{FirstRootAssInclusion}.
\end{proof}

\section{\textsf{LF}-Covers and Ideals of K\"{o}nig Type}

\subsection{The Main Theorem}

\subsubsection{Statement \& Preliminaries}

The notion of ideals of K\"{o}nig type was introduced by Herzog, Hibi, and Moradi in \cite{herzogGradedIdealsKonig2021} as a generalisation of K\"{o}nig graphs, for which the matching number is equal to the vertex cover number (see \cite[Section 1]{herzogGradedIdealsKonig2021} for the details). These ideal possess some desirable properties; for example, in the case of binomial edge ideals, their Cohen-Macaulayness is independent of the characteristic of the base field (see \cite[Corollary 3.8]{herzogGradedIdealsKonig2021}). There are several ways to define these ideals, but for binomial edge ideals there is a particularly elegant characterisation. We first introduce some terminology and notation:

\begin{definition}
	We say that a graph is a \emph{linear forest} if every connected component is a path. For a graph $G$, we say that a linear forest in $G$ is \emph{maximal} if no other linear forest in $G$ has a greater number of edges.
\end{definition}

\begin{note}
	In \cite{herzogGradedIdealsKonig2021} and \cite{laclairInvariantsBinomialEdge2023}, the term \emph{semi-path} is used in place of linear forest, however linear forest appears the more commonly used graph-theoretic term, and so we adopt it here.
\end{note}

\begin{notation}
	For a graph $G$, we denote by $\LF(G)$ the number of edges of a maximal linear forest in $G$.
\end{notation}

\begin{definition}\thmCite[Theorem 3.5]{herzogGradedIdealsKonig2021}
	Let $G$ be a graph on $n$ vertices. We say that $\mathcal{J}(G)$ is of \emph{K\"{o}nig type} if and only if
	\begin{equation*}
		\grade_R(\mathcal{J}(G))=\LF(G)
	\end{equation*}
\end{definition}

The left hand side of this equality is stated in \cite[Theorem 3.5]{herzogGradedIdealsKonig2021} as $2n-\dim(R/\mathcal{J}(G))$, but this is equal to $\grade_R(\mathcal{J}(G))$ by \cref{DimGradeRemark}.

\pagebreak

Ideals of K\"{o}nig type have several desirable properties. In the case of binomial edge ideals of K\"{o}nig type, one such property is that their Cohen-Macaulayness is independent of the characteristic of the base field $k$ (\cite[Corollary 3.8]{herzogGradedIdealsKonig2021}). It is of interest then to characterise these ideals.

To this end, we introduce the notion of $\LF$-covers:

\begin{definition}\label{LFCovDef}
	Let $G$ be a graph. If there exists a maximal linear forest $F$ in $G$, and a set $S\subseteq V(F)$ such that:
	\begin{enumerate}[ref=(\arabic*)]
		\item\label{LFCovDefCrit1} No vertex in $S$ is a leaf of $F$.
		\item\label{LFCovDefCrit2} No two vertices in $S$ are adjacent in $F$.
		\item\label{LFCovDefCrit3} For every edge $\{i,j\}\in E(G)$, one of the following holds:
			\begin{enumerate}[label=\roman*)]
				\item\label{LFCovDefCrit3i} At least one of $i$ or $j$ belongs to $S$. In this case, we say that this vertex \emph{covers} $\{i,j\}$.
				\item\label{LFCovDefCrit3ii} Both $i$ and $j$ belong to the same connected component of $F\setminus S$.
			\end{enumerate}
	\end{enumerate}
	then we say that $(F,S)$ is an \emph{$\LF$-cover} of $G$, and that $G$ is \emph{$\LF$-coverable}.
\end{definition}

This notion is similar to that introduced in \cite[Lemma 5.3]{laclairInvariantsBinomialEdge2023}, with our final criterion expressed differently, however the techniques we use are very different to those of \cite{laclairInvariantsBinomialEdge2023}.

To illustrate this concept, we consider several examples:

\begin{example}
	Let
	\begin{equation*}
		G=\quad\begin{tikzpicture}[x=0.85cm,y=0.85cm,every node/.style={circle,draw=black,fill=black,inner sep=0pt,minimum size=5pt},label distance=0.15cm,line width=0.25mm,baseline={([yshift=-0.5ex]current bounding box.center)}]
        	\node(1)[label=90:$1$] at (0,0) {};
			\node(2)[label=105:$2$] at ([shift=(-135:1)]1) {};
			\node(3)[label=75:$3$] at ([shift=(-45:1)]1) {};
			\node(4)[label=0:$4$] at ([shift=(-45:1)]2) {};
			\node[label=180:$5$](5) at ([shift=(180:1)]2) {};
			\node(6)[label=0:$6$] at ([shift=(0:1)]3) {};
			\node[label=-135:$7$](7) at ([shift=(-135:1)]4) {};
			\node(8)[label=-45:$8$] at ([shift=(-45:1)]4) {};

	    	\foreach \from/\to in {1/2,1/3,2/4,2/5,3/4,3/6,4/7,4/8}
        		\draw[-] (\from) -- (\to);
    	\end{tikzpicture}
	\end{equation*}
	Then $G$ has a unique maximal linear forest
	\begin{equation*}
		F=\quad\begin{tikzpicture}[x=0.85cm,y=0.85cm,every node/.style={circle,draw=black,fill=black,inner sep=0pt,minimum size=5pt},label distance=0.15cm,line width=0.25mm,baseline={([yshift=-0.5ex]current bounding box.center)}]
        	\node(1)[label=90:$1$] at (0,0) {};
			\node(2)[label=105:$2$] at ([shift=(-135:1)]1) {};
			\node(3)[label=75:$3$] at ([shift=(-45:1)]1) {};
			\node(4)[label=0:$4$] at ([shift=(-45:1)]2) {};
			\node[label=180:$5$](5) at ([shift=(180:1)]2) {};
			\node(6)[label=0:$6$] at ([shift=(0:1)]3) {};
			\node[label=-135:$7$](7) at ([shift=(-135:1)]4) {};
			\node(8)[label=-45:$8$] at ([shift=(-45:1)]4) {};

	    	\foreach \from/\to in {1/2,1/3,2/5,3/6,4/7,4/8}
        		\draw[-] (\from) -- (\to);
    	\end{tikzpicture}
	\end{equation*}
	and $\LF$-covers $(F,S)$ for
	\begin{equation*}
		S\in\{\{4\},\{1,4\},\{2,3\},\{2,4\},\{3,4\},\{2,3,4\}\}
	\end{equation*}
	In particular, the minimal $S$ with respect to containment (here circled) are
	\begin{equation*}
		\begin{tikzpicture}[x=0.85cm,y=0.85cm,every node/.style={circle,draw=black,fill=black,inner sep=0pt,minimum size=5pt},label distance=0.15cm,line width=0.25mm,baseline={([yshift=-0.5ex]current bounding box.center)}]
        	\node(1)[label=90:$1$] at (0,0) {};
			\node(2)[label=105:$2$] at ([shift=(-135:1)]1) {};
			\node(3)[label=75:$3$] at ([shift=(-45:1)]1) {};
			\node(4)[label=0:$4$] at ([shift=(-45:1)]2) {};
			\node[label=180:$5$](5) at ([shift=(180:1)]2) {};
			\node(6)[label=0:$6$] at ([shift=(0:1)]3) {};
			\node[label=-135:$7$](7) at ([shift=(-135:1)]4) {};
			\node(8)[label=-45:$8$] at ([shift=(-45:1)]4) {};

			\node[circle,draw,inner sep=3.5pt,line width=1pt,fill=none] at (4){};

	    	\foreach \from/\to in {1/2,1/3,2/5,3/6,4/7,4/8}
        		\draw[-] (\from) -- (\to);
    	\end{tikzpicture}\quad\quad
		\begin{tikzpicture}[x=0.85cm,y=0.85cm,every node/.style={circle,draw=black,fill=black,inner sep=0pt,minimum size=5pt},label distance=0.15cm,line width=0.25mm,baseline={([yshift=-0.5ex]current bounding box.center)}]
        	\node(1)[label=90:$1$] at (0,0) {};
			\node(2)[label=105:$2$] at ([shift=(-135:1)]1) {};
			\node(3)[label=75:$3$] at ([shift=(-45:1)]1) {};
			\node(4)[label=0:$4$] at ([shift=(-45:1)]2) {};
			\node[label=180:$5$](5) at ([shift=(180:1)]2) {};
			\node(6)[label=0:$6$] at ([shift=(0:1)]3) {};
			\node[label=-135:$7$](7) at ([shift=(-135:1)]4) {};
			\node(8)[label=-45:$8$] at ([shift=(-45:1)]4) {};

			\node[circle,draw,inner sep=3.5pt,line width=1pt,fill=none] at (2){};
			\node[circle,draw,inner sep=3.5pt,line width=1pt,fill=none] at (3){};

	    	\foreach \from/\to in {1/2,1/3,2/5,3/6,4/7,4/8}
        		\draw[-] (\from) -- (\to);
    	\end{tikzpicture}
	\end{equation*}
	and so such minimal $S$ need not contain the same number of vertices.

	To see that these are $\LF$-covers, note that $\{2,4\}$ and $\{3,4\}$ are the only edges of $G\setminus F$. In the first example, both $\{2,4\}$ and $\{3,4\}$ are covered by $4$, whilst in the second example, $\{2,4\}$ is covered by $2$, and $\{3,4\}$ is covered by $3$. In both examples, no two vertices in $S$ are adjacent in $F$, and no vertex in $S$ is a leaf of $F$.
\end{example}

\begin{example}
	$(P_n,\varnothing)$ is an $\LF$-cover of $C_n$ for any $n\geq3$, because both endpoints of the only remaining edge $\{1,n\}$ in $C_n\setminus P_n$ belong to the same connected component of $P_n\setminus\varnothing$ (this example is also trivially true in the cases that $n=1$ or $n=2$, since then $C_n=P_n$).
\end{example}

\pagebreak

\begin{example}
	Let
	\begin{equation*}
		G=\quad\begin{tikzpicture}[x=0.85cm,y=0.85cm,every node/.style={circle,draw=black,fill=black,inner sep=0pt,minimum size=5pt},label distance=0.15cm,line width=0.25mm,baseline={([yshift=-0.5ex]current bounding box.center)}]
        	\node(1)[label=117:$1$] at (0,0) {};
			\node(2)[label=90:$2$] at ([shift=(54:1)]1) {};
			\node(3)[label=90:$3$] at ([shift=(-18:1)]2) {};
			\node(4)[label=-90:$4$] at ([shift=(-90:1)]3) {};
			\node[label=-90:$5$](5) at ([shift=(-162:1)]4) {};
			\node(6)[label=180:$6$] at ([shift=(180:1)]1) {};
			\node[label=0:$7$](7) at ([shift=(36:1)]3) {};
			\node(8)[label=0:$8$] at ([shift=(-36:1)]4) {};

	    	\foreach \from/\to in {1/2,1/5,1/6,2/3,3/4,3/7,4/5,4/8}
        		\draw[-] (\from) -- (\to);
    	\end{tikzpicture}
	\end{equation*}
	It is easily checked that, up to isomorphism, the maximal linear forests in $G$ are
	\begin{equation*}
		F_1=\quad\begin{tikzpicture}[x=0.85cm,y=0.85cm,every node/.style={circle,draw=black,fill=black,inner sep=0pt,minimum size=5pt},label distance=0.15cm,line width=0.25mm,baseline={([yshift=-0.5ex]current bounding box.center)}]
        	\node(1)[label=117:$1$] at (0,0) {};
			\node(2)[label=90:$2$] at ([shift=(54:1)]1) {};
			\node(3)[label=90:$3$] at ([shift=(-18:1)]2) {};
			\node(4)[label=-90:$4$] at ([shift=(-90:1)]3) {};
			\node[label=-90:$5$](5) at ([shift=(-162:1)]4) {};
			\node(6)[label=180:$6$] at ([shift=(180:1)]1) {};
			\node[label=0:$7$](7) at ([shift=(36:1)]3) {};
			\node(8)[label=0:$8$] at ([shift=(-36:1)]4) {};

	    	\foreach \from/\to in {1/2,1/5,2/3,3/7,4/5,4/8}
        		\draw[-] (\from) -- (\to);
    	\end{tikzpicture}\quad\quad F_2=\quad\begin{tikzpicture}[x=0.85cm,y=0.85cm,every node/.style={circle,draw=black,fill=black,inner sep=0pt,minimum size=5pt},label distance=0.15cm,line width=0.25mm,baseline={([yshift=-0.5ex]current bounding box.center)}]
        	\node(1)[label=117:$1$] at (0,0) {};
			\node(2)[label=90:$2$] at ([shift=(54:1)]1) {};
			\node(3)[label=90:$3$] at ([shift=(-18:1)]2) {};
			\node(4)[label=-90:$4$] at ([shift=(-90:1)]3) {};
			\node[label=-90:$5$](5) at ([shift=(-162:1)]4) {};
			\node(6)[label=180:$6$] at ([shift=(180:1)]1) {};
			\node[label=0:$7$](7) at ([shift=(36:1)]3) {};
			\node(8)[label=0:$8$] at ([shift=(-36:1)]4) {};

	    	\foreach \from/\to in {1/2,1/6,2/3,3/7,4/5,4/8}
        		\draw[-] (\from) -- (\to);
    	\end{tikzpicture}
	\end{equation*}
	Note that $(F_1,S_1)$ and $(F_2,S_2)$ are $\LF$ covers of $G$ if and only if
	\begin{equation*}
		S_1,S_2\in\{\{1,3\},\{1,4\},\{1,3,4\}\}
	\end{equation*}
	It is not a coincidence that these collections agree: we will see in \cref{MainLFThm} that, for any graph $G$, $(F,S)$ being an $\LF$-cover of $G$ for some $S\subseteq V(G)$ is independent of the choice of maximal linear forest $F$.
\end{example}

Unfortunately, not all graphs are $\LF$-coverable (we will see more such examples in \cref{NonLFCovSubsection}):
\begin{example}\label{NetNoLFCovExample}
	Let
	\begin{equation*}
		G=\quad\begin{tikzpicture}[x=0.85cm,y=0.85cm,every node/.style={circle,draw=black,fill=black,inner sep=0pt,minimum size=5pt},label distance=0.15cm,line width=0.25mm,baseline={([yshift=-0.5ex]current bounding box.center)}]
       		\node(1)[label=180:$1$] at (0,0) {};
    		\node(2)[label=-90:$2$] at ([shift=(-120:1)]1) {};
        	\node(3)[label=-90:$3$] at ([shift=(0:1)]2) {};
			\node(4)[label=180:$4$] at ([shift=(90:1)]1) {};
			\node(5)[label=180:$5$] at ([shift=(-150:1)]2) {};
        	\node(6)[label=0:$6$] at ([shift=(-30:1)]3) {};

	  	  	\foreach \from/\to in {1/2,1/3,1/4,2/3,2/5,3/6}
				\draw[-] (\from.center) -- (\to.center);$$
    	\end{tikzpicture}
	\end{equation*}
	This is sometimes called the \emph{net} or \emph{$3$-sunlet}.

	Up to isomorphism, the maximal linear forests in $G$ are
	\begin{equation*}
		F_1=\quad\begin{tikzpicture}[x=0.85cm,y=0.85cm,every node/.style={circle,draw=black,fill=black,inner sep=0pt,minimum size=5pt},label distance=0.15cm,line width=0.25mm,baseline={([yshift=-0.5ex]current bounding box.center)}]
       		\node(1)[label=180:$1$] at (0,0) {};
    		\node(2)[label=-90:$2$] at ([shift=(-120:1)]1) {};
        	\node(3)[label=-90:$3$] at ([shift=(0:1)]2) {};
			\node(4)[label=180:$4$] at ([shift=(90:1)]1) {};
			\node(5)[label=180:$5$] at ([shift=(-150:1)]2) {};
        	\node(6)[label=0:$6$] at ([shift=(-30:1)]3) {};

	  	  	\foreach \from/\to in {1/2,1/3,2/5,3/6}
				\draw[-] (\from.center) -- (\to.center);
    	\end{tikzpicture}\quad\quad F_2=\quad\begin{tikzpicture}[x=0.85cm,y=0.85cm,every node/.style={circle,draw=black,fill=black,inner sep=0pt,minimum size=5pt},label distance=0.15cm,line width=0.25mm,baseline={([yshift=-0.5ex]current bounding box.center)}]
       		\node(1)[label=180:$1$] at (0,0) {};
    		\node(2)[label=-90:$2$] at ([shift=(-120:1)]1) {};
        	\node(3)[label=-90:$3$] at ([shift=(0:1)]2) {};
			\node(4)[label=180:$4$] at ([shift=(90:1)]1) {};
			\node(5)[label=180:$5$] at ([shift=(-150:1)]2) {};
        	\node(6)[label=0:$6$] at ([shift=(-30:1)]3) {};

	  	  	\foreach \from/\to in {1/2,1/4,2/5,3/6}
				\draw[-] (\from.center) -- (\to.center);
    	\end{tikzpicture}
	\end{equation*}
	Suppose that we had an $\LF$-cover $(F_1,S_1)$ of $G$. The only way to cover $\{1,4\}$ would be to include $1$ in $S_1$, since we cannot include $4$ in $S_1$ because it is a leaf of $F_1$. However $2$ and $3$ would then belong to separate connected components of $F_1\setminus S_1$, and so we would need to include either $2$ or $3$ in $S_1$ to cover $\{2,3\}$. But each of these are adjacent to $1$ in $F_1$, and so cannot be added to $S_1$.

	Now suppose that we had an $\LF$-cover $(F_2,S_2)$ of $G$. In order to cover $\{1,3\}$ and $\{2,3\}$, we would need to include both $1$ and $2$ in $S_2$, since we cannot add $3$ to $S_2$ because it is a leaf of $F_2$. However $1$ and $2$ are adjacent in $F_2$, and so cannot both belong to $S_2$.
\end{example}

\begin{note}
	By \hyperref[MainLFThmIndep]{the second claim} of \cref{MainLFThm}, to show that a graph $G$ is not $\LF$-coverable, it suffices to check a single maximal linear forest in $G$.
\end{note}

We can characterise $\LF$-covers algebraically:
\begin{lemma}\label{LFCovAssEquiv}
	Let $G$ be a graph on $n$ vertices, $F$ a maximal linear forest in $G$, and $S\subseteq V(F)$. Then $(F,S)$ is an $\LF$-cover of $G$ if and only if $S\in\mathcal{C}(F)$ and $\mathcal{J}(G)\subseteq P_S(F)$.
\end{lemma}

\begin{proof}
	\hyperref[LFCovDefCrit1]{Criterion 1} and \hyperref[LFCovDefCrit2]{Criterion 2} of \cref{LFCovDef} are equivalent to saying that $S\in\mathcal{C}(F)$ since $F$ is a linear forest. We have that $\mathcal{J}(G)\subseteq P_S(F)$ if and only if $\delta_{i,j}\in P_S(F)$ for every edge $\{i,j\}\in E(G)$, which is clearly equivalent to \hyperref[LFCovDefCrit3]{Criterion 3} in \cref{LFCovDef} by the definition of $P_S(F)$, and the result follows.
\end{proof}

\pagebreak

Our aim is to prove the following:
\begin{tcolorbox}
	\vspace{-1.4\baselineskip}
	\begin{theorem}\label{MainLFThm}
		Let $G$ be a graph on $n$ vertices. Then $\mathcal{J}(G)$ is of K\"{o}nig type if and only if $G$ is $\LF$-coverable.

		\label{MainLFThmIndep}Furthermore, $(F,S)$ being an $\LF$-cover of $G$ for some $S\subseteq V(G)$ is independent of the choice of maximal linear forest $F$.
	\end{theorem}
\end{tcolorbox}

\hyperref[MainLFThmProof]{The proof} will be given in \cref{MainLFThmProofSubsection}.

We begin by showing part of what is proved in \cite[Lemma 3.3]{herzogGradedIdealsKonig2021}, although we do so avoiding the use of initial ideals:

\begin{proposition}\label{LinearForestRS}
	Let $G$ be a graph on $n$ vertices, and let $e_1,\ldots,e_t$ denote the elements of $\mathcal{J}(G)$ associated to the edges of $G$. Then the $e_i$ form a regular sequence if and only if $G$ is a linear forest.
\end{proposition}

\begin{proof}
	Note first that being a linear forest is equivalent to not containing any cycle $C_m$ for $3\leq m\leq n$ or $K_{1,3}$ (the star with $3$ edges, sometimes called the claw). Since $R$ is graded and the $e_i$ are homogeneous of positive degree, if they form a regular sequence then any permutation of this sequence will remain a regular sequence (see, for example, \cite[Theorem 16.3]{matsumuraCommutativeRingTheory1986}), and so to show necessity it is enough to show that $\delta_{1,2},\delta_{2,3},\ldots,\delta_{m-1,m},\delta_{1,m}$ and $\delta_{1,2},\delta_{1,3},\delta_{1,4}$ do not form regular sequences (the second condition need not be checked when $n\leq3$).

	Firstly, we have
	\begin{align*}
		(\mathcal{J}(P_m):\delta_{1,m})&=\hspace{-0.1cm}\left(\hspace{0.45cm}\bigcap_{\mathclap{S\in\mathcal{C}(P_m)}}\,P_S(P_m)\hspace{-0.05cm}\right)\hspace{-0.1cm}:\delta_{1,m}\hspace{-0.02cm}=\,\bigcap_{\mathclap{S\in\mathcal{C}(P_m)}}\,(P_S(P_m):\delta_{1,m})=\bigcap_{\mathclap{\substack{S\in\mathcal{C}(P_m)\\\delta_{1,m}\notin P_S(P_m)}}}\,P_S(P_m)\\
		&=\bigcap_{\mathclap{\substack{S\in\mathcal{C}(P_m)\\S\neq\varnothing}}}\,P_S(P_m)\supsetneq\mathcal{J}(P_m)
	\end{align*}
	since the $P_S(P_m)$ are prime, and so  $\delta_{1,2},\delta_{2,3},\ldots,\delta_{m-1,m},\delta_{1,m}$ cannot be a regular sequence.

	We similarly have that
	\begin{equation*}
		(\mathcal{J}(K_{1,3}):\delta_{1,4})=((x_1,y_1)\cap\mathcal{J}(K_3)):\delta_{1,4}=((x_1,y_1):\delta_{1,4})\cap(\mathcal{J}(K_3):\delta_{1,4})=\mathcal{J}(K_3)\supsetneq\mathcal{J}(K_{1,3})
	\end{equation*}
	and so $\delta_{1,2},\delta_{1,3},\delta_{1,4}$ cannot be a regular sequence either. This proves necessity.

	For sufficiency, we will first prove by induction that $\delta_{1,2},\delta_{2,3},\ldots,\delta_{m-1,m}$ is a regular sequence for $m\geq 2$.

	The case $m=2$ is trivial since $R$ is a domain.

	We now proceed by induction, and so assume that $\delta_{1,2},\delta_{2,3},\ldots,\delta_{m-2,m-1}$ is a regular sequence.

	As before, we have that
	\begin{equation*}
		(\mathcal{J}(P_{m-1}):\delta_{m-1,m})=\bigcap_{\mathclap{\substack{S\in\mathcal{C}(P_{m-1})\\\delta_{m-1,m}\notin P_S(P_{m-1})}}}\,P_S(P_{m-1})=\bigcap_{\mathclap{S\in\mathcal{C}(P_{m-1})}}\,P_S(P_{m-1})=\mathcal{J}(P_{m-1})
	\end{equation*}
	since clearly $\delta_{m-1,m}\notin P_S(P_{m-1})$ for any $S\in\mathcal{C}(P_{m-1})$, and so $\delta_{1,2},\delta_{2,3},\ldots,\delta_{m-1,m}$ forms a regular sequence as desired.

Joining regular sequences from multiple disjoint paths will still result in a regular sequence, since they have no variables in common. Then we have shown sufficiency, and so we are done.
\end{proof}

\pagebreak

\subsubsection{\textsf{LF}-Covers \& Local Cohomology}\label{LFCovLCSubsection}

\begin{tcolorbox}
	Throughout this subsection, we assume that $k$ is of prime characteristic $p>0$.
\end{tcolorbox}

The next stage of our proof of \cref{MainLFThm} will make use of prime characteristic techniques. Our aim is to relate $\LF$-covers to the associated primes of a certain local cohomology module, which we do in \cref{LocCohoAssPrimeThm}.

We must introduce one final tool beforehand. We begin with a definition:

\begin{definition}
	Let $R$ be a (commutative) Noetherian ring, $I$ an ideal of $R$, and $M$ a finitely generated $R$-module. A sequence of elements $r_1,\ldots,r_t\in I$ is called \emph{$I$-filter regular on $M$} if, for all
	\begin{equation*}
		\mathfrak{p}\in\Supp_R(M)\setminus V(I)
	\end{equation*}
	and $i\leq t$ such that $r_1,\ldots,r_i\in\mathfrak{p}$, we have that $\frac{r_1}{1},\ldots,\frac{r_i}{1}$ is a (possibly improper) $M_\mathfrak{p}$-sequence.

	When $M=R$, we simply say that such a sequence is \emph{$I$-filter regular}.
\end{definition}

In particular, a regular sequence contained in $I$ is $I$-filter regular.

\begin{note}
	Unlike regular sequences, for any ideal $I$ and $R$-module $M$ there exist $I$-filter regular sequences on $M$ of arbitrary length (see, for example, \cite[Proposition 2.2]{asadollahiResultsAssociatedPrimes2003}).
\end{note}

\begin{theorem}[The Nagel-Schenzel Isomorphism]\thmCite[Lemma 3.4]{nagelCohomologicalAnnihilatorsCastelnuovoMumford1994}\label{NSIsom}
	Let $R$ be a (commutative) Noetherian ring, $I$ an ideal of $R$, and $M$ a finitely generated $R$-module. Furthermore, let $r_1,\ldots,r_t$ be an $I$-filter regular sequence on $M$, and set $\mathfrak{a}=(r_1,\ldots,r_t)$. Then
	\begin{equation*}
		H_I^i(M)\cong
		\begin{cases}
			H_\mathfrak{a}^i(M) & \text{\normalfont if }0\leq i<t \\
			H_I^{i-t}(H_\mathfrak{a}^t(M)) & \text{\normalfont if }i\geq t
		\end{cases}
	\end{equation*}
	In particular, setting $i=t$, we have
	\begin{equation*}
		H_I^t(M)\cong H_I^0(H_\mathfrak{a}^t(M))=\Gamma_I(H_\mathfrak{a}^t(M))
	\end{equation*}
	(where $\Gamma_I=H_I^0$ denotes the $I$-torsion functor).
\end{theorem}

\begin{note}
	\nameref{NSIsom} is stated in \thmCite[Lemma 3.4]{nagelCohomologicalAnnihilatorsCastelnuovoMumford1994} only for local rings $(R,\mathfrak{m})$ and $\mathfrak{m}$-filter regular sequences on $R$-modules. \cite[Proposition 2.3]{asadollahiResultsAssociatedPrimes2003} naturally generalises this to general $I$-filter regular sequences on $R$-modules. In both of these papers, the proof of the isomorphism makes use of spectral sequences. For an elementary proof, and of a slightly more general result, see \cite[Theorem 2.7]{huongNotesFrobeniusTest2019}.
\end{note}

The crux of our proof of \cref{LocCohoAssPrimeThm} is as follows:

\begin{lemma}\label{RootLemma}
	Let $G$ be a graph on $n$ vertices, and set $\mathfrak{g}=\mathcal{J}(G)$. Furthermore, let $F$ be a maximal linear forest in $G$, and set $\mathfrak{f}=\mathcal{J}(F)$. Then $(\mathfrak{f}:\mathfrak{g})/\mathfrak{f}$ is a root of $H_\mathfrak{g}^d(R)$, where $d=\smallAbs{E(F)}$.
\end{lemma}

\begin{proof}\label{RootLemmaProof}
	Let $e_1,\ldots,e_d$ be the elements of $\mathfrak{f}$ corresponding to the edges of $F$, which generate $\mathfrak{f}$. By \cref{LinearForestRS}, the $e_i$ form a regular sequence. In particular, they form a $\mathfrak{g}$-filter regular sequence since $\mathfrak{f}\subseteq\mathfrak{g}$. Then, by \hyperref[NSIsom]{the Nagel-Schenzel Isomorphism} and \cref{LCFFinFMod}, we have
	\begin{align*}
		H_\mathfrak{g}^d(R)\cong\Gamma_\mathfrak{g}(H_\mathfrak{f}^d(R))&\cong\Gamma_\mathfrak{g}\hspace{-0.1cm}\left(\dirlim\hspace{-0.05cm}\left[R/\mathfrak{f}\xhra{u\cdot}R/\mathfrak{f}^{\hspace{0.035cm}[p]}\xhra{u^p\cdot}R/\mathfrak{f}^{\hspace{0.035cm}[p^2]}\xhra{u^{p^2}\cdot}\cdots\right]\right)\\
		&\cong\dirlim\hspace{-0.05cm}\left[\Gamma_\mathfrak{g}(R/\mathfrak{f})\xhra{u\cdot}\Gamma_\mathfrak{g}(R/\mathfrak{f}^{\hspace{0.035cm}[p]})\xhra{u^p\cdot}\Gamma_\mathfrak{g}(R/\mathfrak{f}^{\hspace{0.035cm}[p^2]})\xhra{u^{p^2}\cdot}\cdots\right]\\
		&\cong\dirlim\hspace{-0.05cm}\left[(\mathfrak{f}:\mathfrak{g}^\infty)/\mathfrak{f}\xhra{u\cdot}(\mathfrak{f}^{\hspace{0.035cm}[p]}:\mathfrak{g}^\infty)/\mathfrak{f}^{\hspace{0.035cm}[p]}\xhra{u^p\cdot}(\mathfrak{f}^{\hspace{0.035cm}[p^2]}:\mathfrak{g}^\infty)/\mathfrak{f}^{\hspace{0.035cm}[p^2]}\xhra{u^{p^2}\cdot}\cdots\right]
	\end{align*}
	where $u=(e_1\cdots e_d)^{p-1}$.

\pagebreak

	The injectivity of the first map follows since the $e_i$ form a regular sequence (see, for example, \cite[Theorem 3.2]{ocarrollGeneralizedFractionsSharp1983}), and the injectivity of the subsequent maps is a result of the exactness of the Frobenius functor (which is due to \nameref{Kunz} since $R$ is regular) and the left exactness of $\Gamma_\mathfrak{g}$.

	Since $\mathfrak{f}=\mathcal{J}(F)$ is radical by \cref{BEIRadical}, we have $(\mathfrak{f}:\mathfrak{g}^\infty)=(\mathfrak{f}:\mathfrak{g})$, and so we are done.
\end{proof}

The following lemma is not difficult (and holds over any ring), but we include a proof for completeness:

\begin{lemma}\label{AssLemma}
	Let $\mathfrak{a}$ be a radical ideal of $R$ with minimal primary decomposition
	\begin{equation*}
		\mathfrak{a}=\bigcap_{\mathclap{i=1}}^{t}\,\mathfrak{p}_i
	\end{equation*}
	for some $\mathfrak{p}_i\in\Spec(R)$ and $t\geq1$. Furthermore, let
	\begin{equation*}
		\mathfrak{b}=\bigcap_{\mathclap{i=1}}^{m}\,\mathfrak{p}_i
	\end{equation*}
	for some $1\leq m<t$. Then
	\begin{equation*}
		\Ass_R(\mathfrak{b}/\mathfrak{a})=\{\mathfrak{p}_i:m+1\leq i\leq t\}
	\end{equation*}
\end{lemma}

\begin{proof}
	Let $b\in\mathfrak{b}$, and suppose that $rb\in\mathfrak{a}$ for some $r\in R$. Then $rb\in\mathfrak{p}_i$ for all $1\leq i\leq t$. Since the $\mathfrak{p}_i$ are prime, we must have that $r\in\mathfrak{p}_i$ for each $m+1\leq i\leq t$ such that $b\notin\mathfrak{p}_i$. Then
	\begin{equation*}
		\Ann_R(b+\mathfrak{a})\subseteq\,\bigcap_{\mathclap{\substack{i=m+1\\b\notin\mathfrak{p}_i}}}^t\,\mathfrak{p}_i
	\end{equation*}
	The converse clearly holds, and this ideal is only prime when $b\notin\mathfrak{p}_i$ for exactly one $m+1\leq i\leq t$, so
	\begin{equation*}
		\Ass_R(\mathfrak{b}/\mathfrak{a})\subseteq\{\mathfrak{p}_i:m+1\leq i\leq t\}
	\end{equation*}
	Since the primary decomposition of $\mathfrak{a}$ is minimal, no $\mathfrak{p}_i$ is redundant. Then we can find
	\begin{equation*}
		b_j\in\bigcap_{\mathclap{\substack{i=1\\i\neq j}}}^t\,\mathfrak{p}_i\setminus\mathfrak{p}_j
	\end{equation*}
	for each $m+1\leq j\leq t$, and the result follows.
\end{proof}

We can now prove the main result of this subsection:

\begin{theorem}\label{LocCohoAssPrimeThm}
	Let $G$ be a graph on $n$ vertices, and set $\mathfrak{g}=\mathcal{J}(G)$. Furthermore, let $F$ be a maximal linear forest in $G$ with $\smallAbs{E(F)}=d$. Then
	\begin{equation*}
		\Ass_R(H_\mathfrak{g}^d(R))=\{P_S(F):\text{$S\in\mathcal{C}(G)$ such that $(F,S)$ is an $\LF$-cover of $G$}\}
	\end{equation*}
\end{theorem}

\begin{proof}
	Let $\mathfrak{f}=\mathcal{J}(F)$. We have that
	\begin{equation*}
		(\mathfrak{f}:\mathfrak{g})=\hspace{-0.05cm}\left(\hspace{0.35cm}\bigcap_{\mathclap{S\in\mathcal{C}(F)}}\,P_S(F)\hspace{-0.05cm}\right)\hspace{-0.075cm}:\mathfrak{g}=\hspace{0.1cm}\bigcap_{\mathclap{S\in\mathcal{C}(F)}}\,(P_S(F):\mathfrak{g})=\bigcap_{\mathclap{\substack{S\in\mathcal{C}(F)\\\mathfrak{g}\nsubseteq P_S(F)}}}\,P_S(F)
	\end{equation*}
	since each $P_S(F)$ is prime. Since $(\mathfrak{f}:\mathfrak{g})/\mathfrak{f}$ is a root of $H_\mathfrak{g}^d(R)$ by \cref{RootLemma}, their associated primes are equal by \cref{RootAssEqual}. By \cref{AssLemma} and \cref{LFCovAssEquiv}, these primes are exactly as claimed.
\end{proof}

\begin{note}
	At the time of writing, we do not know if \cref{LocCohoAssPrimeThm} holds in characteristic $0$.
\end{note}

\pagebreak

\subsubsection{Proof of the Main Theorem}\label{MainLFThmProofSubsection}

\begin{tcolorbox}
	We now allow $k$ to be of any characteristic again.
\end{tcolorbox}

Whilst we have only proven \cref{LocCohoAssPrimeThm} in prime characteristic, this will allow us to prove \cref{MainLFThm} in arbitrary characteristic:

\begin{proposition}\label{LFCovImpliesKonig}
	Let $G$ be a graph on $n$ vertices, set $\mathfrak{g}=\mathcal{J}(G)$, and let $F$ be a maximal linear forest in $G$. Then if $G$ has an $\LF$-cover $(F,S)$ for some $S\subseteq V(G)$, $\mathfrak{g}$ is of K\"{o}nig type.
\end{proposition}

\begin{proof}
	Let $d=|E(F)|$. Suppose first that $k$ has prime characteristic $p>0$. Since $(F,S)$ is an $\LF$-cover of $G$, then by \cref{LocCohoAssPrimeThm} we know that $\Ass_R(H_\mathfrak{g}^d(R))$ is non-empty. In particular, we have that $H_\mathfrak{g}^d(R)\neq0$, and so $\grade_R(\mathfrak{g})\leq d$ (see, for example, \cite[Theorem 6.2.7]{brodmannLocalCohomologyAlgebraic2013}). But, by \cref{DimGradeRemark}, $\grade_R(\mathfrak{g})$ is independent of the characteristic of $k$, and so this inequality holds in any case.

	Conversely, the elements of $\mathfrak{g}$ corresponding to the edges of $F$ form a regular sequence by \cref{LinearForestRS}, so $\grade_R(\mathfrak{g})\geq d$ and we are done.
\end{proof}

\begin{proposition}\label{KonigImpliesLFCov}
	Let $G$ be a graph on $n$ vertices, set $\mathfrak{g}=\mathcal{J}(G)$, and suppose that $\mathfrak{g}$ is of K\"{o}nig type. Then $G$ is $\LF$-coverable.
\end{proposition}

\begin{proof}
	Let $F$ be a maximal linear forest in $G$, and set $\mathfrak{f}=\mathcal{J}(F)$. Since $\mathfrak{g}$ is of K\"{o}nig type, and $\mathfrak{f}$ is trivially of K\"{o}nig type, we have
	\begin{equation*}
		\grade_R(\mathfrak{g})=\LF(G)=\LF(F)=\grade_R(\mathfrak{f})
	\end{equation*}
	and so
	\begin{equation*}
		\dim(R/\mathfrak{g})=\dim(R)-\grade_R(\mathfrak{g})=\dim(R)-\grade_R(\mathfrak{f})=\dim(R/\mathfrak{f})
	\end{equation*}
	by \cref{DimGradeRemark}.

	Let $d=\dim(R/\mathfrak{g})$. Then we can find $\mathfrak{p}_0,\ldots,\mathfrak{p}_d\in\Spec(R)$ such that
	\begin{equation*}
		\mathfrak{f}\subseteq\mathfrak{g}\subseteq\mathfrak{p}_0\subsetneq\cdots\subsetneq\mathfrak{p}_d
	\end{equation*}
	Now, $\mathfrak{p}_0$ must be minimal over $\mathfrak{f}$, since otherwise we would have $\dim(R/\mathfrak{f})>d$, and so $\mathfrak{p}_0\in\Ass_R(\mathfrak{f})$. Then $\mathfrak{p}_0=P_S(F)$ for some $S\subseteq V(F)$, and we are done by \cref{LFCovAssEquiv}.
\end{proof}

\begin{customProof}[Proof of \cref{MainLFThm}]\label{MainLFThmProof}
	The first claim follows immediately from \cref{LFCovImpliesKonig} and \cref{KonigImpliesLFCov}.

	We will now prove the second claim, so suppose that $(F,S)$ is an $\LF$-cover of $G$, and take another maximal linear forest $F'$ in $G$. When $k$ has prime characteristic $p>0$, we have
	\begin{equation*}
		P_S(F)\in\Ass_R(H_\mathfrak{g}^d(R))
	\end{equation*}
	by \cref{LocCohoAssPrimeThm}, where $d=\LF(G)$ and $\mathfrak{g}=\mathcal{J}(G)$. This local cohomology module does not depend on the choice of maximal linear forest, and so, again by \cref{LocCohoAssPrimeThm}, we must have $P_S(F)=P_{S'}(F')$ for some $S'\subseteq V(F')$ such that $(F',S')$ is an $\LF$-cover of $G$. By \hyperref[PSGDef]{definition}, $x_i,y_i\in P_S(F)$ if and only if $i\in S$, and similarly $x_i,y_i\in P_{S'}(F')$ if and only if $i\in S'$. Then since $P_S(F)=P_{S'}(F')$, we must have $S=S'$, and so $(F',S)$ is an $\LF$-cover of $G$. Since $\LF$-covers can be characterised purely combinatorially, and are therefore independent of characteristic, the result follows.
\end{customProof}

\pagebreak

\subsection{Some Classes of \textsf{LF}-Coverable Graphs}\label{LFCovGraphsSection}

We will next show that traceable, complete bipartite and trivially perfect graphs are $\LF$-coverable. We will show the same for trees, by way of constructing an algorithm for calculating an explicit $\LF$-cover.

We also conjecture that several well-known classes of graphs are $\LF$-coverable, with computational evidence to support this.

\begin{definition}
	If a graph $G$ on $n$ vertices contains $P_n$, we say that $G$ is \emph{traceable}. If it contains $C_n$, we say it that is \emph{Hamiltonian}.
\end{definition}

We recover, phrased differently, a similar result to one direction of \cite[Proposition 5.2]{ohtaniGraphsIdealsGenerated2011}:

\begin{proposition}\label{TraceableLFCov}
	Traceable graphs are $\LF$-coverable.
\end{proposition}

\begin{proof}
	Take $F$ to be a path connecting all the vertices of the graph, and set $S=\varnothing$.
\end{proof}

\begin{note}
	It is shown in \cite{posaHamiltonianCircuitsRandom1976} that almost all graphs are Hamiltonian. In particular, they are traceable, and so almost all binomial edge ideals have grade $n-1$.
\end{note}

\begin{proposition}
	Complete bipartite graphs $K_{a,b}$ are $\LF$-coverable for all $a,b\geq1$.
\end{proposition}

\begin{proof}
	We may assume without loss of generality that $a\geq b$. Let $v_1,\ldots,v_a$ and $w_1,\ldots,w_b$ be the vertices in each partition.

	If $a=b$ then take
	\begin{equation*}
		F=\{\{v_1,w_1\},\{w_1,v_2\},\{v_2,w_2\},\ldots,\{v_a,w_b\}\}
	\end{equation*}
	and $S=\varnothing$.

	If $a>b$ then take
	\begin{equation*}
		F=\{\{v_1,w_1\},\{w_1,v_2\},\{v_2,w_2\},\ldots,\{v_a,w_b\},\{w_b,v_{a+1}\}\}
	\end{equation*}
	and $S=\{w_1,\ldots,w_b\}$.
\end{proof}

\begin{definition}
	We say that a graph $G$ is \emph{trivially perfect} if it can be constructed inductively, starting from a single vertex $K_1$, via either the disjoint union of smaller trivially perfect graphs, or the addition of a vertex adjacent to all other vertices of a smaller trivially perfect graph (that is, taking the join $\ast$ of a smaller trivially perfect graph with $K_1$).
\end{definition}

\begin{proposition}\label{TrivPerfLFCov}
	Trivially perfect graphs are $\LF$-coverable.
\end{proposition}

\begin{proof}
	If a graph $G_1$ has an $\LF$-cover $(F_1,S_1)$, and a graph $G_2$ has an $\LF$-cover $(F_2,S_2)$, then $G_1\sqcup G_2$ clearly has an $\LF$-cover given by $F=F_1\sqcup F_2$ and $S=S_1\cup S_2$. Then we may assume that $G$ is connected.

	The graph $K_1$ trivially has an $\LF$-cover, and so we may proceed by induction on $\smallAbs{V(G)}$.

	We can obtain any connected trivially perfect graph $G$ with more than one vertex by starting with a smaller trivially perfect graph $G'$, then taking $G=G'\ast K_1$. We have an $\LF$-cover $(F',S')$ of $G'$ by the inductive hypothesis.

	If $F'$ consists of a single connected component then it must be a path, so $G'$ is traceable. Clearly the join of any traceable graph with a single point remains traceable, and so we are done by \cref{TraceableLFCov}.
	
	If $F'$ consists of more than one connected component, then certainly two leaves from separate components, say $l_1$ and $l_2$, become joined in $G$. If we say $V(K_1)=\{v\}$, then we may extend $F'$ by joining these two connected components along $\{l_1,v\}$ and $\{v,l_2\}$ to create a new linear forest $F$. Since $l_1$ and $l_2$ are leaves in $F$, we may set $S=S'\cup\{v\}$. This will cover all the new edges added by $v$ in $G$, and so $(F,S)$ is an $\LF$-cover of $G$.
\end{proof}

\pagebreak

There are several related classes of graphs which we conjecture are also $\LF$-coverable:\label{LFCovGraphConjs}

\begin{definition}
	We say that a graph $G$ is a \emph{cograph} if it contains no induced $P_4$.
\end{definition}

\begin{conjecture}\label{CographConj}
	Cographs are $\LF$-coverable.
\end{conjecture}
We have verified this conjecture using \citeMacaulay{} for all cographs with up to $13$ vertices.

Trivially perfect graphs can be alternatively characterised as those with no induced $P_4$ or $C_4$ (see \cite[Theorem 3]{jing-hoQuasithresholdGraphs1996}), and so form a subclass of cographs. Then \cref{CographConj} implies \cref{TrivPerfLFCov}.

\tcbrule{}
\vspace{-0.5\baselineskip}

\begin{definition}
	Let $\rho=(\sigma_1,\ldots,\sigma_m)$ be a permutation of the integers $1,\ldots,m$. Then we define the \emph{permutation graph} $G$ associated to $\rho$ by setting $V(G)=\{1,\ldots,m\}$ and
	\begin{equation*}
		E(G)=\{\{v,w\}:\text{\normalfont$\sigma_w$ appears before $\sigma_v$ in $\rho$, $1\leq v<w\leq m$}\}
	\end{equation*}
\end{definition}

\begin{conjecture}\label{PermGraphConj}
	Permutation graphs are $\LF$-coverable.
\end{conjecture}
We have verified this conjecture using \citeMacaulay{} for all permutation graphs with up to $11$ vertices.

Since permutation graphs are a superclass of cographs (see \cite[pp.280--281]{bosePatternMatchingPermutations1998}), \cref{PermGraphConj} implies \cref{CographConj}.

\tcbrule{}
\vspace{-0.5\baselineskip}

\begin{definition}
	Let $I_1,\ldots,I_m$ be a collection of intervals on the real line. Then we define the \emph{interval graph} $G$ associated to this collection by setting $V(G)=\{1,\ldots,m\}$ and
	\begin{equation*}
		E(G)=\{\{v,w\}:\text{$I_v\cap I_w\neq\varnothing$, $1\leq v<w\leq m$}\}
	\end{equation*}
\end{definition}

\begin{conjecture}\label{IntGraphConj}
	Interval graphs are $\LF$-coverable.
\end{conjecture}
We have verified this conjecture using \citeMacaulay{} for all interval graphs with up to $10$ vertices.

Interestingly, trivially perfect graphs are exactly the graphs which are both cographs and interval graphs (again, see \cite[Theorem 3]{jing-hoQuasithresholdGraphs1996}), and so \cref{IntGraphConj} implies \cref{TrivPerfLFCov} also.

\tcbrule{}

Before \hyperref[TreeAlgoProof]{proving} that trees are $\LF$-coverable, we first introduce a definition and some notation:

\begin{definition}\label{EBPDef}
	Let $T$ be a tree and $b$ and branch point of $T$. We say that $b$ is \emph{extremal} if it is adjacent to at most one other branch point.
\end{definition}

\begin{note}
	\cref{EBPDef} is not established terminology.
\end{note}

\begin{notation}
	For a tree $T$ and vertex $v\in V(T)$, we denote by $L_T(v)$ the set of vertices of $T$ which are leaves adjacent to $v$, and set $L_T[v]\defeq L_T(v)\cup\{v\}$.
\end{notation}

\begin{proposition}\label{EBPExistence}
	Let $T$ be a tree on $n$ vertices, where $n\geq3$. Then $T$ has an extremal branch point.
\end{proposition}

\begin{proof}
	The proof will be by induction on $n$. The case $n=3$ is clear, so suppose that $n>3$. Remove a leaf $l$ and apply the inductive hypothesis to $T\setminus\{l\}$, so $T\setminus\{l\}$ has an extremal branch point $b$. When $l$ is reconnected, if it is adjacent to a leaf $l'$ of $b$ then $l'$ is now an extremal branch point. Otherwise $b$ will remain an extremal branch point, and so we are done. 
\end{proof}

We can now \hyperref[TreeAlgoProof]{prove} that trees are $\LF$-coverable, and exhibit \hyperref[TreeAlgo]{an algorithm} for efficiently computing an $\LF$-cover of any tree $T$. A distinct method is given in \cite[Algorithm 5.8]{laclairInvariantsBinomialEdge2023} for computing a linear forest and set satisfying their criteria.

These algorithms are of independent interest, since they effectively calculate the dimension of $\mathcal{J}(T)$ by \cref{DimGradeRemark}. For example, an algorithm to compute this dimension is given as a special case of \cite[Theorem 2.2]{masciaKrullDimensionRegularity2020}.

\begin{theorem}\label{TreesLFCov}
	Let $T$ be a tree on $n$ vertices, where $n\geq2$. Then $T$ is $\LF$-coverable, and we can inductively construct a maximal linear forest $F$ in $T$ and a set of vertices $S\subseteq V(F)$ such that $(F,S)$ is an $\LF$-cover of $T$.
\end{theorem}

\begin{proof}\label{TreeAlgoProof}
	The proof will be by induction on $n$. The case $n=2$ is trivial: take $F=T$ and $S=\varnothing$.

	Then suppose that $n>2$, and that the statement of the theorem holds for all trees with fewer than $n$ vertices.

	We may choose an extremal branch point $b$ of $T$ by \cref{EBPExistence}, and consider two cases:
	\begin{enumerate}
		\item\label{TreeAlgoCase1} If $b$ has a single leaf $l$, then $b$ must be adjacent to another branch point $b'$, since $T$ cannot be a star because $n>2$. Let $T'=T\setminus\{l\}$. By the inductive hypothesis, $T'$ has an $\LF$-cover $(F',S')$.

		Let $F=F'\cup\{\{b,l\}\}$, and set $S=S'$. We claim that $(F,S)$ is an $\LF$-cover of $T$. The three conditions of \cref{LFCovDef} are clearly satisfied, so it remains only to show that $F$ is a maximal linear forest in $G$. Since $\{b,b'\}$ must satisfy the conditions of \cref{LFCovDef} for $(F',S')$, we must either have that $\{b,b'\}\in E(F')$, in which case $F$ is a linear forest since we will just be extending the path containing this edge, or $\{b,b'\}\notin E(F')$ and $b'\in S'$, in which case $\{b,l\}$ will be its own connected component in $F$, and so $F$ is a linear forest in this case also. Maximality is clear, since $T$ has one more edge than $T'$, and $F$ has one more edge than $F'$ (which is maximal in $T'$).
		
		\item\label{TreeAlgoCase2} Now suppose that $b$ has at least two leaves $l_1$ and $l_2$, and set $T'=T\setminus L_T[b]$. Again by the inductive hypothesis, $T'$ has an $\LF$-cover $(F',S')$. Note that if $T$ is a star, $T'$ will be the null graph $K_0$, which trivially has $\LF$-cover $(\varnothing,\varnothing)$.

		Let
		\begin{equation*}
			F=F'\cup\{\{b,l_1\},\{b,l_2\}\}
		\end{equation*}
		and set $S=S'\cup\{b\}$. In this case, $\{\{b,l_1\},\{b,l_2\}\}$ will be its own connected component in $F$, since we removed $L_T[b]$ from $T$ when obtaining $F'$, and so $F$ is a linear forest. It is maximal in $T$ since we are adding a star (centred at $v$) to $T'$, and so we can add at most two edges to $F'$ before introducing a claw (that is, $K_{1,3}$). Furthermore, every other leaf edge of $b$, and the edge $\{b,b'\}$ for some branch point $b'\in E(T)$ which will exist if $T$ is not a star, will be covered by $b\in S$, and so $(F,S)$ is an $\LF$-cover of $T$. This concludes the proof.\qedhere
	\end{enumerate}
\end{proof}

\pagebreak

The resulting algorithm is very simple:

\begin{algorithm}\label{TreeAlgo}\mbox{}\\
	\textit{Input:} A tree $T$.
	\\\textit{Output:} An $\LF$-cover $(F,S)$ of $T$.

\tcbdoublerule{}

\begin{lstlisting}
$F\leftarrow K_0$
$S\leftarrow\varnothing$
while $T\neq K_0$ do
£\indentrule£if $T$ has a single edge $e$ then
£\indentrule[2]	£$F\leftarrow F\cup\{e\}$
£\indentrule[2]	£$T\leftarrow K_0$
£\indentrule£else
£\indentrule[2]£$b\leftarrow$ some extremal branch point of $T$
£\indentrule[2]£if $b$ has a single leaf $l$ then
£\indentrule[3]£$F\leftarrow F\cup\{\{b,l\}\}$
£\indentrule[3]£$T\leftarrow T\setminus\{l\}$
£\indentrule[2]£else
£\indentrule[3]£$l_1\leftarrow$ some leaf of $b$
£\indentrule[3]£$l_2\leftarrow$ some leaf of $b$ other than $l_1$
£\indentrule[3]£$F\leftarrow F\cup\{\{b,l_1\},\{b,l_2\}\}$
£\indentrule[3]£$S\leftarrow S\cup\{b\}$
£\indentrule[3]£$T\leftarrow T\setminus L_T[b]$
£\indentrule[2]£end
£\indentrule£end
end
return $(F,S)$
\end{lstlisting}

\vspace{-0.5\baselineskip}
\tcbdoublerule{}

\end{algorithm}

We will now illustrate the operation of \cref{TreeAlgo}:

\begin{example}
	Let
	\begin{equation*}
		T=\quad\begin{tikzpicture}[x=0.675cm,y=0.675cm,every node/.style={circle,draw=black,fill=black,inner sep=0pt,minimum size=5pt},label distance=0.1cm,line width=0.25mm,baseline={([yshift=-0.5ex]current bounding box.center)}]
        	\node[label={0:$1$}](1) at (2,0) {};
       		\node[label={0:$2$}](2) at (2,1) {};
			\node[label={0:$3$}](3) at (2,2) {};
			\node[label={0:$4$}](4) at (1,3) {};
			\node[label={0:$5$}](5) at (3,3) {};
			\node[label={180:$6$}](6) at (0,3) {};
			\node[label={0:$7$}](7) at (1,4) {};
			\node[label={0:$8$}](8) at (3,4) {};
			\node[label={180:$9$}](9) at (0,5) {};
			\node[label={90:$10$}](10) at (1,5) {};
			\node[label={0:$11$}](11) at (2,5) {};

			\foreach \from/\to in {1/2,2/3,3/4,3/5,4/6,4/7,5/8,7/9,7/10,7/11}
        		\draw[-] (\from.center) -- (\to.center);
   		\end{tikzpicture}
	\end{equation*}

	We will compute $F$ and $S$ inductively, showing that we can obtain

	\begin{equation*}
		F=\quad\begin{tikzpicture}[x=0.675cm,y=0.675cm,every node/.style={circle,draw=black,fill=black,inner sep=0pt,minimum size=5pt},label distance=0.1cm,line width=0.25mm,baseline={([yshift=-0.5ex]current bounding box.center)}]
        	\node[label={0:$1$}](1) at (2,0) {};
       		\node[label={0:$2$}](2) at (2,1) {};
			\node[label={0:$3$}](3) at (2,2) {};
			\node[label={0:$4$}](4) at (1,3) {};
			\node[label={0:$5$}](5) at (3,3) {};
			\node[label={180:$6$}](6) at (0,3) {};
			\node[label={0:$7$}](7) at (1,4) {};
			\node[label={0:$8$}](8) at (3,4) {};
			\node[label={180:$9$}](9) at (0,5) {};
			\node[label={0:$11$}](11) at (2,5) {};

			\foreach \from/\to in {1/2,3/4,3/5,4/6,5/8,7/9,7/11}
        		\draw[-] (\from.center) -- (\to.center);
   		\end{tikzpicture}
	\end{equation*}
	and $S=\{3,7\}$ by following \cref{TreeAlgo}.

	\begin{note}
		$F$ and $S$ may depend on choices of extremal branch points and leaves in \cref{TreeAlgo}. In this example there are many alternative choices, as a somewhat trivially different $\LF$-cover we could remove say $\{7,9\}$ and add $\{7,10\}$ in $E(F)$ whilst keeping $S$ the same.
	\end{note}

	We will explain in detail the first two steps of \cref{TreeAlgo} in this example, before presenting every step in \hyperref[LFCovExample]{a table}.

	We begin with the full tree $T$, and select an extremal branch point. Here we choose $b=7$ (as indicated by a square):
	\begin{equation*}
		T=\quad
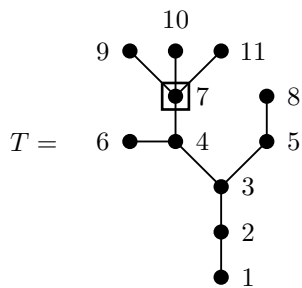
\begin{tikzpicture}[x=0.6cm,y=0.6cm,every node/.style={circle,draw=black,fill=black,inner sep=0pt,minimum size=5pt},label distance=0.1cm,line width=0.25mm,baseline={([yshift=-0.5ex]current bounding box.center)}]
       	 	\node[label={0:$1$}](1) at (2,0) {};
			\node[label={0:$2$}](2) at (2,1) {};
			\node[label={0:$3$}](3) at (2,2) {};
			\node[label={0:$4$}](4) at (1,3) {};
			\node[label={0:$5$}](5) at (3,3) {};
			\node[label={180:$6$}](6) at (0,3) {};
			\node[label={0:$7$}](7) at (1,4) {};
			\node[label={0:$8$}](8) at (3,4) {};
			\node[label={180:$9$}](9) at (0,5) {};
			\node[label={90:$10$}](10) at (1,5) {};
			\node[label={0:$11$}](11) at (2,5) {};

			\node[rectangle,draw,inner sep=5pt,line width=1pt,fill=none] at (7){};

			\foreach \from/\to in {1/2,2/3,3/4,3/5,4/6,4/7,5/8,7/9,7/10,7/11}
        		\draw[-] (\from.center) -- (\to.center);
    	\end{tikzpicture}
	\end{equation*}
	This puts us in \hyperref[TreeAlgoCase2]{Case 2} of \hyperref[TreeAlgoProof]{the proof} of \cref{TreesLFCov}. With notation as in \hyperref[TreeAlgoCase2]{that case}, we choose $l_1=9$ and $l_2=11$. Then the edges $\{7,9\}$ and $\{7,11\}$ will be added to the linear forest $F$, and $7$ will be added to $S$ (as indicated by a circle):
	\begin{equation*}
		F=\quad
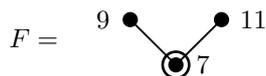
\begin{tikzpicture}[x=0.6cm,y=0.6cm,every node/.style={circle,draw=black,fill=black,inner sep=0pt,minimum size=5pt},label distance=0.1cm,line width=0.25mm,baseline={([yshift=-0.5ex]current bounding box.center)}]
			\node[label={0:$7$}](7) at (1,4) {};
			\node[label={180:$9$}](9) at (0,5) {};
			\node[label={0:$11$}](11) at (2,5) {};

			\node[circle,draw,inner sep=3.5pt,line width=1pt,fill=none] at (7){};

			\foreach \from/\to in {7/9,7/11}
    			\draw[-] (\from.center) -- (\to.center);
    	\end{tikzpicture}
	\end{equation*}
	Note that $7$ then covers $\{4,7\}$ and $\{7,10\}$.

	We then remove $7$ and its leaves from $T$. This leaves us with
	\begin{equation*}
		T=\quad
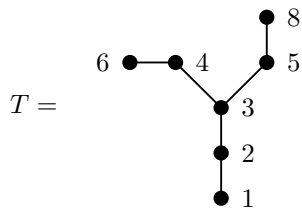
\begin{tikzpicture}[x=0.6cm,y=0.6cm,every node/.style={circle,draw=black,fill=black,inner sep=0pt,minimum size=5pt},label distance=0.1cm,line width=0.25mm,baseline={([yshift=-0.5ex]current bounding box.center)}]
  		 	\node[label={0:$1$}](1) at (2,0) {};
	 		\node[label={0:$2$}](2) at (2,1) {};
			\node[label={0:$3$}](3) at (2,2) {};
			\node[label={0:$4$}](4) at (1,3) {};
			\node[label={0:$5$}](5) at (3,3) {};
			\node[label={180:$6$}](6) at (0,3) {};
			\node[label={0:$8$}](8) at (3,4) {};

			\foreach \from/\to in {1/2,2/3,3/4,3/5,4/6,5/8}
    			\draw[-] (\from.center) -- (\to.center);
    	\end{tikzpicture}
	\end{equation*}
	for which we must continue calculating an $\LF$-cover.

	Then we next select the extremal branch point $b=4$:
	\begin{equation*}
		T=\quad
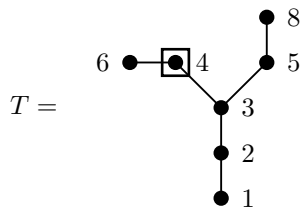
\begin{tikzpicture}[x=0.6cm,y=0.6cm,every node/.style={circle,draw=black,fill=black,inner sep=0pt,minimum size=5pt},label distance=0.1cm,line width=0.25mm,baseline={([yshift=-0.5ex]current bounding box.center)}]
       		 	\node[label={0:$1$}](1) at (2,0) {};
		 		\node[label={0:$2$}](2) at (2,1) {};
				\node[label={0:$3$}](3) at (2,2) {};
				\node[label={0:$4$}](4) at (1,3) {};
				\node[label={0:$5$}](5) at (3,3) {};
				\node[label={180:$6$}](6) at (0,3) {};
				\node[label={0:$8$}](8) at (3,4) {};

				\node[rectangle,draw,inner sep=5pt,line width=1pt,fill=none] at (4){};

				\foreach \from/\to in {1/2,2/3,3/4,3/5,4/6,5/8}
        			\draw[-] (\from.center) -- (\to.center);
    		\end{tikzpicture}
		\end{equation*}
		We are then in \hyperref[TreeAlgoCase1]{Case 1} of \hyperref[TreeAlgoProof]{the proof} of \cref{TreesLFCov}. With notation as in \hyperref[TreeAlgoCase1]{that case}, we then have $l=6$, so $\{4,6\}$ will be added to the linear forest $F$:

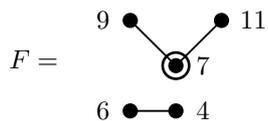
		\begin{equation*}
			F=\quad\begin{tikzpicture}[x=0.6cm,y=0.6cm,every node/.style={circle,draw=black,fill=black,inner sep=0pt,minimum size=5pt},label distance=0.1cm,line width=0.25mm,baseline={([yshift=-0.5ex]current bounding box.center)}]
				\node[label={0:$4$}](4) at (1,3) {};
				\node[label={180:$6$}](6) at (0,3) {};
				\node[label={0:$7$}](7) at (1,4) {};
				\node[label={180:$9$}](9) at (0,5) {};
				\node[label={0:$11$}](11) at (2,5) {};

				\node[circle,draw,inner sep=3.5pt,line width=1pt,fill=none] at (7){};

				\foreach \from/\to in {4/6,7/9,7/11}
        			\draw[-] (\from.center) -- (\to.center);
    		\end{tikzpicture}
		\end{equation*}
		We then remove $6$ from $T$. This leaves us with
		\begin{equation*}
		T=\quad
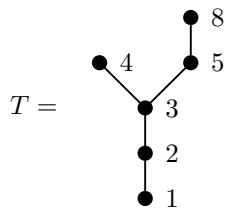
\begin{tikzpicture}[x=0.6cm,y=0.6cm,every node/.style={circle,draw=black,fill=black,inner sep=0pt,minimum size=5pt},label distance=0.1cm,line width=0.25mm,baseline={([yshift=-0.5ex]current bounding box.center)}]
  		 	\node[label={0:$1$}](1) at (2,0) {};
	 		\node[label={0:$2$}](2) at (2,1) {};
			\node[label={0:$3$}](3) at (2,2) {};
			\node[label={0:$4$}](4) at (1,3) {};
			\node[label={0:$5$}](5) at (3,3) {};
			\node[label={0:$8$}](8) at (3,4) {};

			\foreach \from/\to in {1/2,2/3,3/4,3/5,5/8}
    			\draw[-] (\from.center) -- (\to.center);
    	\end{tikzpicture}
	\end{equation*}
	for which we must continue calculating an $\LF$-cover.

\pagebreak

	The full execution of \cref{TreeAlgo} for this example is as follows (the light grey vertices and edges are simply placeholders to allow us to visualise each step in the context of the original tree):

	\label{LFCovExample}
	\begin{center}	
	\begin{longtable}{ | >{\centering\arraybackslash} m{1cm} | >{\centering\arraybackslash} m{4.25cm} | >{\centering\arraybackslash} m{4.25cm} |  >{\centering\arraybackslash} m{4.25cm} | }
		\hline
		\emph{Step} & \emph{Current $T$} & \emph{Resulting $F$} & \emph{Remaining $T$}
		\\ \hline
		1
		&
		\adjustbox{margin=5pt}{
			\begin{tikzpicture}[x=0.6cm,y=0.6cm,every node/.style={circle,draw=black,fill=black,inner sep=0pt,minimum size=5pt},label distance=0.1cm,line width=0.25mm]
       		 	\node[label={0:$1$}](1) at (2,0) {};
		 		\node[label={0:$2$}](2) at (2,1) {};
				\node[label={0:$3$}](3) at (2,2) {};
				\node[label={0:$4$}](4) at (1,3) {};
				\node[label={0:$5$}](5) at (3,3) {};
				\node[label={180:$6$}](6) at (0,3) {};
				\node[label={0:$7$}](7) at (1,4) {};
				\node[label={0:$8$}](8) at (3,4) {};
				\node[label={180:$9$}](9) at (0,5) {};
				\node[label={90:$10$}](10) at (1,5) {};
				\node[label={0:$11$}](11) at (2,5) {};

				\node[rectangle,draw,inner sep=5pt,line width=1pt,fill=none] at (7){};

				\foreach \from/\to in {1/2,2/3,3/4,3/5,4/6,4/7,5/8,7/9,7/10,7/11}
        			\draw[-] (\from.center) -- (\to.center);
    		\end{tikzpicture}
		}
		&
		\adjustbox{margin=5pt}{
			\begin{tikzpicture}[x=0.6cm,y=0.6cm,every node/.style={circle,draw=black,fill=black,inner sep=0pt,minimum size=5pt},label distance=0.1cm,line width=0.25mm]
				\node[label={0:$7$}](7) at (1,4) {};
				\node[label={180:$9$}](9) at (0,5) {};
				\node[label={0:$11$}](11) at (2,5) {};

				\node[draw=lightgray,fill=lightgray](1) at (2,0) {};
		 		\node[draw=lightgray,fill=lightgray](2) at (2,1) {};
				\node[draw=lightgray,fill=lightgray](3) at (2,2) {};
				\node[draw=lightgray,fill=lightgray](4) at (1,3) {};
				\node[draw=lightgray,fill=lightgray](5) at (3,3) {};
				\node[draw=lightgray,fill=lightgray](6) at (0,3) {};
				\node[draw=lightgray,fill=lightgray](8) at (3,4) {};
				\node[draw=lightgray,fill=lightgray](10) at (1,5) {};

				\node[circle,draw,inner sep=3.5pt,line width=1pt,fill=none] at (7){};

				\foreach \from/\to in {7/9,7/11}
        			\draw[-] (\from.center) -- (\to.center);

				\begin{scope}[on background layer]
					\foreach \from/\to in {1/2,2/3,3/4,3/5,4/6,4/7,5/8,7/10}
        				\draw[-,lightgray] (\from.center) -- (\to.center);
				\end{scope}
    		\end{tikzpicture}
		}
		&
		\adjustbox{margin=5pt}{
			\begin{tikzpicture}[x=0.6cm,y=0.6cm,every node/.style={circle,draw=black,fill=black,inner sep=0pt,minimum size=5pt},label distance=0.1cm,line width=0.25mm]
       		 	\node[label={0:$1$}](1) at (2,0) {};
		 		\node[label={0:$2$}](2) at (2,1) {};
				\node[label={0:$3$}](3) at (2,2) {};
				\node[label={0:$4$}](4) at (1,3) {};
				\node[label={0:$5$}](5) at (3,3) {};
				\node[label={180:$6$}](6) at (0,3) {};
				\node[label={0:$8$}](8) at (3,4) {};

				\node[draw=lightgray,fill=lightgray](7) at (1,4) {};
				\node[draw=lightgray,fill=lightgray](9) at (0,5) {};
				\node[draw=lightgray,fill=lightgray](10) at (1,5) {};
				\node[draw=lightgray,fill=lightgray](11) at (2,5) {};

				\foreach \from/\to in {1/2,2/3,3/4,3/5,4/6,5/8}
        			\draw[-] (\from.center) -- (\to.center);

				\begin{scope}[on background layer]
					\foreach \from/\to in {4/7,7/9,7/10,7/11}
        				\draw[-,lightgray] (\from.center) -- (\to.center);
				\end{scope}
    		\end{tikzpicture}
		}
		\\ \hline
		2
		&
		\adjustbox{margin=5pt}{
			\begin{tikzpicture}[x=0.6cm,y=0.6cm,every node/.style={circle,draw=black,fill=black,inner sep=0pt,minimum size=5pt},label distance=0.1cm,line width=0.25mm]
       		 	\node[label={0:$1$}](1) at (2,0) {};
		 		\node[label={0:$2$}](2) at (2,1) {};
				\node[label={0:$3$}](3) at (2,2) {};
				\node[label={0:$4$}](4) at (1,3) {};
				\node[label={0:$5$}](5) at (3,3) {};
				\node[label={180:$6$}](6) at (0,3) {};
				\node[label={0:$8$}](8) at (3,4) {};

				\node[draw=lightgray,fill=lightgray](7) at (1,4) {};
				\node[draw=lightgray,fill=lightgray](9) at (0,5) {};
				\node[draw=lightgray,fill=lightgray](10) at (1,5) {};
				\node[draw=lightgray,fill=lightgray](11) at (2,5) {};

				\node[rectangle,draw,inner sep=5pt,line width=1pt,fill=none] at (4){};

				\foreach \from/\to in {1/2,2/3,3/4,3/5,4/6,5/8}
        			\draw[-] (\from.center) -- (\to.center);

				\begin{scope}[on background layer]
					\foreach \from/\to in {4/7,7/9,7/10,7/11}
        				\draw[-,lightgray] (\from.center) -- (\to.center);
				\end{scope}
    		\end{tikzpicture}
		}
		&
		\adjustbox{margin=5pt}{
			\begin{tikzpicture}[x=0.6cm,y=0.6cm,every node/.style={circle,draw=black,fill=black,inner sep=0pt,minimum size=5pt},label distance=0.1cm,line width=0.25mm]
				\node[label={0:$4$}](4) at (1,3) {};
				\node[label={180:$6$}](6) at (0,3) {};
				\node[label={0:$7$}](7) at (1,4) {};
				\node[label={180:$9$}](9) at (0,5) {};
				\node[label={0:$11$}](11) at (2,5) {};

				\node[draw=lightgray,fill=lightgray](1) at (2,0) {};
		 		\node[draw=lightgray,fill=lightgray](2) at (2,1) {};
				\node[draw=lightgray,fill=lightgray](3) at (2,2) {};
				\node[draw=lightgray,fill=lightgray](5) at (3,3) {};
				\node[draw=lightgray,fill=lightgray](8) at (3,4) {};
				\node[draw=lightgray,fill=lightgray](10) at (1,5) {};

				\node[circle,draw,inner sep=3.5pt,line width=1pt,fill=none] at (7){};

				\foreach \from/\to in {4/6,7/9,7/11}
        			\draw[-] (\from.center) -- (\to.center);

				\begin{scope}[on background layer]
					\foreach \from/\to in {1/2,2/3,3/4,3/5,4/7,5/8,7/10}
        				\draw[-,lightgray] (\from.center) -- (\to.center);
				\end{scope}
    		\end{tikzpicture}
		}
		&
		\adjustbox{margin=5pt}{
			\begin{tikzpicture}[x=0.6cm,y=0.6cm,every node/.style={circle,draw=black,fill=black,inner sep=0pt,minimum size=5pt},label distance=0.1cm,line width=0.25mm]
       		 	\node[label={0:$1$}](1) at (2,0) {};
		 		\node[label={0:$2$}](2) at (2,1) {};
				\node[label={0:$3$}](3) at (2,2) {};
				\node[label={0:$4$}](4) at (1,3) {};
				\node[label={0:$5$}](5) at (3,3) {};
				\node[label={0:$8$}](8) at (3,4) {};

				\node[draw=lightgray,fill=lightgray](6) at (0,3) {};
				\node[draw=lightgray,fill=lightgray](7) at (1,4) {};
				\node[draw=lightgray,fill=lightgray](9) at (0,5) {};
				\node[draw=lightgray,fill=lightgray](10) at (1,5) {};
				\node[draw=lightgray,fill=lightgray](11) at (2,5) {};

				\foreach \from/\to in {1/2,2/3,3/4,3/5,5/8}
        			\draw[-] (\from.center) -- (\to.center);

				\begin{scope}[on background layer]
					\foreach \from/\to in {4/6,4/7,7/9,7/10,7/11}
        				\draw[-,lightgray] (\from.center) -- (\to.center);
				\end{scope}
    		\end{tikzpicture}
		}
		\\ \hline
		3
		&
		\adjustbox{margin=5pt}{
			\begin{tikzpicture}[x=0.6cm,y=0.6cm,every node/.style={circle,draw=black,fill=black,inner sep=0pt,minimum size=5pt},label distance=0.1cm,line width=0.25mm]
       		 	\node[label={0:$1$}](1) at (2,0) {};
		 		\node[label={0:$2$}](2) at (2,1) {};
				\node[label={0:$3$}](3) at (2,2) {};
				\node[label={0:$4$}](4) at (1,3) {};
				\node[label={0:$5$}](5) at (3,3) {};
				\node[label={0:$8$}](8) at (3,4) {};

				\node[draw=lightgray,fill=lightgray](6) at (0,3) {};
				\node[draw=lightgray,fill=lightgray](7) at (1,4) {};
				\node[draw=lightgray,fill=lightgray](9) at (0,5) {};
				\node[draw=lightgray,fill=lightgray](10) at (1,5) {};
				\node[draw=lightgray,fill=lightgray](11) at (2,5) {};

				\node[rectangle,draw,inner sep=5pt,line width=1pt,fill=none] at (5){};

				\foreach \from/\to in {1/2,2/3,3/4,3/5,5/8}
        			\draw[-] (\from.center) -- (\to.center);

				\begin{scope}[on background layer]
					\foreach \from/\to in {4/6,4/7,7/9,7/10,7/11}
        				\draw[-,lightgray] (\from.center) -- (\to.center);
				\end{scope}
    		\end{tikzpicture}
		}
		&
		\adjustbox{margin=5pt}{
			\begin{tikzpicture}[x=0.6cm,y=0.6cm,every node/.style={circle,draw=black,fill=black,inner sep=0pt,minimum size=5pt},label distance=0.1cm,line width=0.25mm]
				\node[label={0:$4$}](4) at (1,3) {};
				\node[label={0:$5$}](5) at (3,3) {};
				\node[label={180:$6$}](6) at (0,3) {};
				\node[label={0:$7$}](7) at (1,4) {};
				\node[label={0:$8$}](8) at (3,4) {};
				\node[label={180:$9$}](9) at (0,5) {};
				\node[label={0:$11$}](11) at (2,5) {};

				\node[draw=lightgray,fill=lightgray](1) at (2,0) {};
		 		\node[draw=lightgray,fill=lightgray](2) at (2,1) {};
				\node[draw=lightgray,fill=lightgray](3) at (2,2) {};
				\node[draw=lightgray,fill=lightgray](10) at (1,5) {};

				\node[circle,draw,inner sep=3.5pt,line width=1pt,fill=none] at (7){};

				\foreach \from/\to in {4/6,5/8,7/9,7/11}
        			\draw[-] (\from.center) -- (\to.center);

				\begin{scope}[on background layer]
					\foreach \from/\to in {1/2,2/3,3/4,3/5,4/7,7/10}
        				\draw[-,lightgray] (\from.center) -- (\to.center);
				\end{scope}
    		\end{tikzpicture}
		}
		&
		\adjustbox{margin=5pt}{
			\begin{tikzpicture}[x=0.6cm,y=0.6cm,every node/.style={circle,draw=black,fill=black,inner sep=0pt,minimum size=5pt},label distance=0.1cm,line width=0.25mm]
       		 	\node[label={0:$1$}](1) at (2,0) {};
		 		\node[label={0:$2$}](2) at (2,1) {};
				\node[label={0:$3$}](3) at (2,2) {};
				\node[label={0:$4$}](4) at (1,3) {};
				\node[label={0:$5$}](5) at (3,3) {};

				\node[draw=lightgray,fill=lightgray](6) at (0,3) {};
				\node[draw=lightgray,fill=lightgray](7) at (1,4) {};
				\node[draw=lightgray,fill=lightgray](8) at (3,4) {};
				\node[draw=lightgray,fill=lightgray](9) at (0,5) {};
				\node[draw=lightgray,fill=lightgray](10) at (1,5) {};
				\node[draw=lightgray,fill=lightgray](11) at (2,5) {};

				\foreach \from/\to in {1/2,2/3,3/4,3/5}
        			\draw[-] (\from.center) -- (\to.center);

				\begin{scope}[on background layer]
					\foreach \from/\to in {4/6,4/7,5/8,7/9,7/10,7/11}
        				\draw[-,lightgray] (\from.center) -- (\to.center);
				\end{scope}
    		\end{tikzpicture}
		}
		\\ \hline
		4
		&
		\adjustbox{margin=5pt}{
			\begin{tikzpicture}[x=0.6cm,y=0.6cm,every node/.style={circle,draw=black,fill=black,inner sep=0pt,minimum size=5pt},label distance=0.1cm,line width=0.25mm]
       		 	\node[label={0:$1$}](1) at (2,0) {};
		 		\node[label={0:$2$}](2) at (2,1) {};
				\node[label={0:$3$}](3) at (2,2) {};
				\node[label={0:$4$}](4) at (1,3) {};
				\node[label={0:$5$}](5) at (3,3) {};

				\node[draw=lightgray,fill=lightgray](6) at (0,3) {};
				\node[draw=lightgray,fill=lightgray](7) at (1,4) {};
				\node[draw=lightgray,fill=lightgray](8) at (3,4) {};
				\node[draw=lightgray,fill=lightgray](9) at (0,5) {};
				\node[draw=lightgray,fill=lightgray](10) at (1,5) {};
				\node[draw=lightgray,fill=lightgray](11) at (2,5) {};

				\node[rectangle,draw,inner sep=5pt,line width=1pt,fill=none] at (3){};

				\foreach \from/\to in {1/2,2/3,3/4,3/5}
        			\draw[-] (\from.center) -- (\to.center);

				\begin{scope}[on background layer]
					\foreach \from/\to in {4/6,4/7,5/8,7/9,7/10,7/11}
        				\draw[-,lightgray] (\from.center) -- (\to.center);
				\end{scope}
    		\end{tikzpicture}
		}
		&
		\adjustbox{margin=5pt}{
			\begin{tikzpicture}[x=0.6cm,y=0.6cm,every node/.style={circle,draw=black,fill=black,inner sep=0pt,minimum size=5pt},label distance=0.1cm,line width=0.25mm]
				\node[label={0:$3$}](3) at (2,2) {};
				\node[label={0:$4$}](4) at (1,3) {};
				\node[label={0:$5$}](5) at (3,3) {};
				\node[label={180:$6$}](6) at (0,3) {};
				\node[label={0:$7$}](7) at (1,4) {};
				\node[label={0:$8$}](8) at (3,4) {};
				\node[label={180:$9$}](9) at (0,5) {};
				\node[label={0:$11$}](11) at (2,5) {};

				\node[draw=lightgray,fill=lightgray](1) at (2,0) {};
				\node[draw=lightgray,fill=lightgray](2) at (2,1) {};
				\node[draw=lightgray,fill=lightgray](10) at (1,5) {};

				\node[circle,draw,inner sep=3.5pt,line width=1pt,fill=none] at (3){};
				\node[circle,draw,inner sep=3.5pt,line width=1pt,fill=none] at (7){};

				\foreach \from/\to in {3/4,3/5,4/6,5/8,7/9,7/11}
        			\draw[-] (\from.center) -- (\to.center);

				\begin{scope}[on background layer]
					\foreach \from/\to in {1/2,2/3,4/7,7/10}
        				\draw[-,lightgray] (\from.center) -- (\to.center);
				\end{scope}
    		\end{tikzpicture}
		}
		&
		\adjustbox{margin=5pt}{
			\begin{tikzpicture}[x=0.6cm,y=0.6cm,every node/.style={circle,draw=black,fill=black,inner sep=0pt,minimum size=5pt},label distance=0.1cm,line width=0.25mm]
       		 	\node[label={0:$1$}](1) at (2,0) {};
		 		\node[label={0:$2$}](2) at (2,1) {};

				\node[draw=lightgray,fill=lightgray](3) at (2,2) {};
				\node[draw=lightgray,fill=lightgray](4) at (1,3) {};
				\node[draw=lightgray,fill=lightgray](5) at (3,3) {};
				\node[draw=lightgray,fill=lightgray](6) at (0,3) {};
				\node[draw=lightgray,fill=lightgray](7) at (1,4) {};
				\node[draw=lightgray,fill=lightgray](8) at (3,4) {};
				\node[draw=lightgray,fill=lightgray](9) at (0,5) {};
				\node[draw=lightgray,fill=lightgray](10) at (1,5) {};
				\node[draw=lightgray,fill=lightgray](11) at (2,5) {};

				\foreach \from/\to in {1/2}
        			\draw[-] (\from.center) -- (\to.center);

				\begin{scope}[on background layer]
					\foreach \from/\to in {2/3,3/4,3/5,4/6,4/7,5/8,7/9,7/10,7/11}
        				\draw[-,lightgray] (\from.center) -- (\to.center);
				\end{scope}
    		\end{tikzpicture}
		}
		\\ \hline
		5
		&
		\adjustbox{margin=5pt}{
			\begin{tikzpicture}[x=0.6cm,y=0.6cm,every node/.style={circle,draw=black,fill=black,inner sep=0pt,minimum size=5pt},label distance=0.1cm,line width=0.25mm]
       		 	\node[label={0:$1$}](1) at (2,0) {};
		 		\node[label={0:$2$}](2) at (2,1) {};

				\node[draw=lightgray,fill=lightgray](3) at (2,2) {};
				\node[draw=lightgray,fill=lightgray](4) at (1,3) {};
				\node[draw=lightgray,fill=lightgray](5) at (3,3) {};
				\node[draw=lightgray,fill=lightgray](6) at (0,3) {};
				\node[draw=lightgray,fill=lightgray](7) at (1,4) {};
				\node[draw=lightgray,fill=lightgray](8) at (3,4) {};
				\node[draw=lightgray,fill=lightgray](9) at (0,5) {};
				\node[draw=lightgray,fill=lightgray](10) at (1,5) {};
				\node[draw=lightgray,fill=lightgray](11) at (2,5) {};

				\foreach \from/\to in {1/2}
        			\draw[-] (\from.center) -- (\to.center);

				\begin{scope}[on background layer]
					\foreach \from/\to in {2/3,3/4,3/5,4/6,4/7,5/8,7/9,7/10,7/11}
        				\draw[-,lightgray] (\from.center) -- (\to.center);
				\end{scope}
    		\end{tikzpicture}
		}
		&
		\adjustbox{margin=5pt}{
			\begin{tikzpicture}[x=0.6cm,y=0.6cm,every node/.style={circle,draw=black,fill=black,inner sep=0pt,minimum size=5pt},label distance=0.1cm,line width=0.25mm]
				\node[label={0:$1$}](1) at (2,0) {};
				\node[label={0:$2$}](2) at (2,1) {};
				\node[label={0:$3$}](3) at (2,2) {};
				\node[label={0:$4$}](4) at (1,3) {};
				\node[label={0:$5$}](5) at (3,3) {};
				\node[label={180:$6$}](6) at (0,3) {};
				\node[label={0:$7$}](7) at (1,4) {};
				\node[label={0:$8$}](8) at (3,4) {};
				\node[label={180:$9$}](9) at (0,5) {};
				\node[label={0:$11$}](11) at (2,5) {};

				\node[draw=lightgray,fill=lightgray](10) at (1,5) {};

				\node[circle,draw,inner sep=3.5pt,line width=1pt,fill=none] at (3){};
				\node[circle,draw,inner sep=3.5pt,line width=1pt,fill=none] at (7){};

				\foreach \from/\to in {1/2,3/4,3/5,4/6,5/8,7/9,7/11}
        			\draw[-] (\from.center) -- (\to.center);

				\begin{scope}[on background layer]
					\foreach \from/\to in {2/3,4/7,7/10}
        				\draw[-,lightgray] (\from.center) -- (\to.center);
				\end{scope}
    		\end{tikzpicture}
		}
		&
		\adjustbox{margin=5pt}{
			\begin{tikzpicture}[x=0.6cm,y=0.6cm,every node/.style={circle,draw=black,fill=black,inner sep=0pt,minimum size=5pt},label distance=0.1cm,line width=0.25mm]
				\node[draw=lightgray,fill=lightgray](1) at (2,0) {};
		 		\node[draw=lightgray,fill=lightgray](2) at (2,1) {};
				\node[draw=lightgray,fill=lightgray](3) at (2,2) {};
				\node[draw=lightgray,fill=lightgray](4) at (1,3) {};
				\node[draw=lightgray,fill=lightgray](5) at (3,3) {};
				\node[draw=lightgray,fill=lightgray](6) at (0,3) {};
				\node[draw=lightgray,fill=lightgray](7) at (1,4) {};
				\node[draw=lightgray,fill=lightgray](8) at (3,4) {};
				\node[draw=lightgray,fill=lightgray](9) at (0,5) {};
				\node[draw=lightgray,fill=lightgray](10) at (1,5) {};
				\node[draw=lightgray,fill=lightgray](11) at (2,5) {};

				\begin{scope}[on background layer]
					\foreach \from/\to in {1/2,2/3,3/4,3/5,4/6,4/7,5/8,7/9,7/10,7/11}
        				\draw[-,lightgray] (\from.center) -- (\to.center);
				\end{scope}
			\end{tikzpicture}
		}
		\\ \hline
	\end{longtable}
\end{center}

	\vspace{-2\baselineskip}
	Then we have calculated $F$ and $S$ as claimed, and so by \cref{MainLFThm} we have
	\begin{equation*}
		\grade_R(\mathcal{J}(T))=\smallAbs{E(F)}=7
	\end{equation*}
\end{example}

\pagebreak

\subsection{Some Graphs Without \textsf{LF}-covers}\label{NonLFCovSubsection}

\begin{example}
	For $n=6$, we saw directly in \cref{NetNoLFCovExample} that
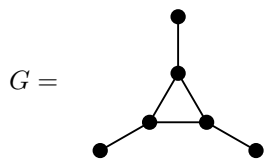
	\begin{equation*}
		G=\quad\begin{tikzpicture}[x=0.75cm,y=0.75cm,every node/.style={circle,draw=black,fill=black,inner sep=0pt,minimum size=5pt},label distance=0.15cm,line width=0.25mm,baseline={([yshift=-0.5ex]current bounding box.center)}]
       		\node(1) at (0,0) {};
    		\node(2) at ([shift=(-120:1)]1) {};
        	\node(3) at ([shift=(0:1)]2) {};
			\node(4) at ([shift=(90:1)]1) {};
			\node(5) at ([shift=(-150:1)]2) {};
        	\node(6) at ([shift=(-30:1)]3) {};

	  	  	\foreach \from/\to in {1/2,1/3,1/4,2/3,2/5,3/6}
				\draw[-] (\from.center) -- (\to.center);
    	\end{tikzpicture}
	\end{equation*}
	is not $\LF$-coverable. It can be checked using \citeMacaulay{} that $\grade_R(\mathcal{J}(G))=5$, but $\LF(G)=4$. This example is minimal in the sense that all graphs with $6$ vertices or fewer are $\LF$-coverable other than this one.
\end{example}

Being bipartite is also not enough to ensure $\LF$-coverability:

\begin{example}
	For $n=10$ and
	\begin{align*}
		G&=\quad\begin{tikzpicture}[x=0.75cm,y=0.75cm,every node/.style={circle,draw=black,fill=black,inner sep=0pt,minimum size=5pt},label distance=0.15cm,line width=0.25mm,baseline={([yshift=-0.5ex]current bounding box.center)}]
       		\node(1) at (0,0) {};
			\node(2) at (1,0) {};
			\node(3) at (2,0) {};
			\node(4) at (3,0) {};
			\node(5) at (4,0) {};
			\node(6) at (0,1) {};
			\node(7) at (1,1) {};
			\node(8) at (2,1) {};
			\node(9) at (3,1) {};
			\node(10) at (4,1) {};
	    	\foreach \from/\to in {1/6,1/7,1/9,2/7,3/7,3/8,3/10,4/8,4/9,5/9,5/10}
        		\draw[-] (\from.center) -- (\to.center);
    	\end{tikzpicture}\\[0.3cm]
		&=\quad\begin{tikzpicture}[x=0.75cm,y=0.4cm,every node/.style={circle,draw=black,fill=black,inner sep=0pt,minimum size=5pt},label distance=0.15cm,line width=0.25mm,baseline={([yshift=-0.5ex]current bounding box.center)}]
       		\node(1) at (1,3) {};
			\node(2) at (0,0) {};
			\node(3) at (2,0) {};
			\node(4) at (2,2) {};
			\node(5) at (3,3) {};
			\node(6) at (0,3) {};
			\node(7) at (1,0) {};
			\node(8) at (2,1) {};
			\node(9) at (2,3) {};
			\node(10) at (3,0) {};
	    	\foreach \from/\to in {1/6,1/7,1/9,2/7,3/7,3/8,3/10,4/8,4/9,5/9,5/10}
        		\draw[-] (\from.center) -- (\to.center);
    	\end{tikzpicture}
	\end{align*}
	it can be checked using \citeMacaulay{} that $\grade_R(\mathcal{J}(G))=9$, but $\LF(G)=8$. This example is minimal in the sense that all bipartite graphs with fewer than $10$ vertices are $\LF$-coverable, and there is only one other bipartite graph on $10$ vertices which is not $\LF$-coverable:
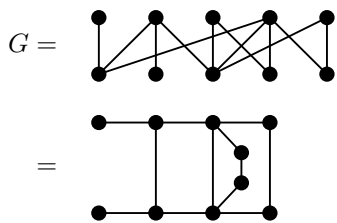
	\begin{align*}
		G&=\quad\begin{tikzpicture}[x=0.75cm,y=0.75cm,every node/.style={circle,draw=black,fill=black,inner sep=0pt,minimum size=5pt},label distance=0.15cm,line width=0.25mm,baseline={([yshift=-0.5ex]current bounding box.center)}]
       		\node(1) at (0,0) {};
			\node(2) at (1,0) {};
			\node(3) at (2,0) {};
			\node(4) at (3,0) {};
			\node(5) at (4,0) {};
			\node(6) at (0,1) {};
			\node(7) at (1,1) {};
			\node(8) at (2,1) {};
			\node(9) at (3,1) {};
			\node(10) at (4,1) {};
	    	\foreach \from/\to in {1/6,1/7,1/9,2/7,3/7,3/8,3/9,3/10,4/8,4/9,5/9,5/10}
        		\draw[-] (\from.center) -- (\to.center);
    	\end{tikzpicture}\\[0.3cm]
		&=\quad\begin{tikzpicture}[x=0.375cm,y=0.4cm,every node/.style={circle,draw=black,fill=black,inner sep=0pt,minimum size=5pt},label distance=0.15cm,line width=0.25mm,baseline={([yshift=-0.5ex]current bounding box.center)}]
       		\node(1) at (2,3) {};
			\node(2) at (0,0) {};
			\node(3) at (4,0) {};
			\node(4) at (5,2) {};
			\node(5) at (6,3) {};
			\node(6) at (0,3) {};
			\node(7) at (2,0) {};
			\node(8) at (5,1) {};
			\node(9) at (4,3) {};
			\node(10) at (6,0) {};
	    	\foreach \from/\to in {1/6,1/7,1/9,2/7,3/7,3/8,3/9,3/10,4/8,4/9,5/9,5/10}
        		\draw[-] (\from.center) -- (\to.center);
    	\end{tikzpicture}
	\end{align*}
	which also has $\grade_R(\mathcal{J}(G))=9$ and $\LF(G)=8$, however this contains an additional edge, and there are no other bipartite examples with $11$ edges or fewer.
\end{example}

In all of the examples we have seen so far, $\grade_R(\mathcal{J}(G))$ and $\LF(G)$ differ by at most $1$. However this is not necessarily the case:

\begin{example}
	For $n=11$ and
	\begin{equation*}
		G=\quad
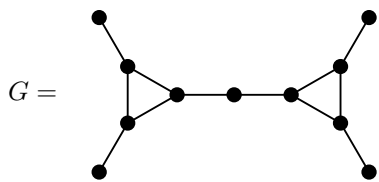
\begin{tikzpicture}[x=0.75cm,y=0.75cm,every node/.style={circle,draw=black,fill=black,inner sep=0pt,minimum size=5pt},label distance=0.15cm,line width=0.25mm,baseline={([yshift=-0.5ex]current bounding box.center)}]
       		\node(1) at (0,0) {};
			\node(2) at ([shift=(180:1)]1) {};
			\node(3) at ([shift=(0:1)]1) {};
			\node(4) at ([shift=(150:1)]2) {};
			\node(5) at ([shift=(-150:1)]2) {};
			\node(6) at ([shift=(30:1)]3) {};
			\node(7) at ([shift=(-30:1)]3) {};
			\node(8) at ([shift=(120:1)]4) {};
			\node(9) at ([shift=(-120:1)]5) {};
			\node(10) at ([shift=(60:1)]6) {};
			\node(11) at ([shift=(-60:1)]7) {};

	 		\foreach \from/\to in {1/2,1/3,2/4,2/5,3/6,3/7,4/5,4/8,5/9,6/7,6/10,7/11}
        		\draw[-] (\from.center) -- (\to.center);
    	\end{tikzpicture}
	\end{equation*}
	it can be checked using \citeMacaulay{} that $\grade_R(\mathcal{J}(G))=10$, but $\LF(G)=8$. This example is minimal in the sense that there are no other such examples (where $\grade_R(\mathcal{J}(G))$ and $\LF(G)$ differ by more than $1$) with $12$ edges or fewer. There is however a single such example with $10$ vertices or fewer, given by
	\begin{equation*}
		G=\quad
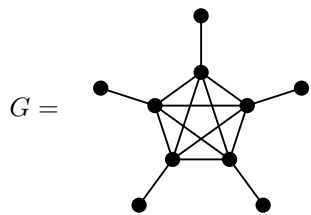
\begin{tikzpicture}[x=0.75cm,y=0.75cm,every node/.style={circle,draw=black,fill=black,inner sep=0pt,minimum size=5pt},label distance=0.15cm,line width=0.25mm,baseline={([yshift=-0.5ex]current bounding box.center)}]
       		\node(1) at (0,0) {};
			\node(2) at ([shift=(-36:1)]1) {};
			\node(3) at ([shift=(-108:1)]2) {};
			\node(4) at ([shift=(180:1)]3) {};
			\node(5) at ([shift=(108:1)]4) {};

			\node(6) at ([shift=(90:1)]1) {};
			\node(7) at ([shift=(18:1)]2) {};
			\node(8) at ([shift=(-54:1)]3) {};
			\node(9) at ([shift=(-126:1)]4) {};
			\node(10) at ([shift=(162:1)]5) {};

	 		\foreach \from/\to in {1/2,1/3,1/4,1/5,1/6,2/3,2/4,2/5,2/7,3/4,3/5,3/8,4/5,4/9,5/10}
        		\draw[-] (\from.center) -- (\to.center);
    	\end{tikzpicture}
	\end{equation*}
	which has $\grade_R(\mathcal{J}(G))=9$ and $\LF(G)=7$ for $n=10$.
\end{example}

\pagebreak

\section{Calculating Explicit Roots}

\cref{RootLemma} suggests that we may be able to explicitly calculate roots of $H_{\mathcal{J}(G)}^d(R)$ when a graph $G$ on $n$ vertices has an $\LF$-cover with $d$ edges. In fact, using \hyperref[NSIsom]{the Nagel-Schenzel Isomorphism}, we could calculate roots of any local cohomology module of the binomial edge ideal of any graph if we could find a $\mathcal{J}(G)$-filter regular sequence of sufficient length.

However, in practice, proving that a sequence of elements is $\mathcal{J}(G)$-filter regular turns out to be quite difficult. Even in the case of $\LF$-coverable graphs with covers with $d$ edges, explicitly calculating roots of $H_{\mathcal{J}(G)}^d(R)$ is not easy.

We perform this calculation here for Hamiltonian graphs. We will first do this for the complete graph $K_n$, and then extend the result to all Hamiltonian graphs.

\begin{tcolorbox}
	We now again assume that $k$ is of prime characteristic $p>0$, and that $n\geq3$.

	Furthermore, we define an $\mathbb{N}^n$-grading on $R$ by setting
	\begin{equation*}
		\deg(x_i)=\deg(y_i)=e_i
	\end{equation*}
	where $e_i$ is the $n$-tuple with $0$ everywhere except for $1$ at the $i$th position.
\end{tcolorbox}

Let $G=K_n$, and set $\mathfrak{g}=\mathcal{J}(G)$. We have that $(P_n,\varnothing)$ is an $\LF$-cover of $K_n$ by \cref{TraceableLFCov} since $K_n$ is traceable, and we set $\mathfrak{a}=\mathcal{J}(P_n)$. We may assume that
\begin{equation*}
	E(P_n)=\{\{1,2\},\ldots,\{n-1,n\}\}
\end{equation*}
By \cref{RootLemma}, we can obtain an explicit root of $H_\mathfrak{g}^{n-1}(R)$ by calculating $(\mathfrak{a}:\mathfrak{g})$.

\begin{notation}
	For $0\leq l\leq n-2$, set
	\begin{equation*}
		z_l\defeq x_2\cdots x_{n-1-l}y_{n-l}\cdots y_{n-1}
	\end{equation*}
\end{notation}
Here $l$ should be thought of as the number of ``$y$''s in this expression.

For example, when $n=5$ we have
\begin{align*}
	z_0&=x_2x_3x_4\\
	z_1&=x_2x_3y_4\\
	z_2&=x_2y_3y_4\\
	z_3&=y_2y_3y_4
\end{align*}
We aim to prove the following:

\begin{theorem}\label{ColonTheorem}
	Let
	\begin{equation*}
		\mathfrak{b}=\mathfrak{a}+(z_0,\ldots,z_{n-2})
	\end{equation*}
	Then $(\mathfrak{a}:\mathfrak{g})=\mathfrak{b}$, and so $\mathfrak{b}/\mathfrak{a}$ is a root of $H_\mathfrak{g}^{n-1}(R)$.
\end{theorem}

To show that $\mathfrak{b}\subseteq(\mathfrak{a}:\mathfrak{g})$, we must prove that, for any $0\leq l\leq n-2$ and $1\leq i<j\leq n$, we have $z_l\delta_{i,j}\in\mathfrak{a}$.

We will do this by proving a slightly stronger statement. 

\begin{notation}
	For any $1\leq i<j\leq n$ and $0\leq l\leq j-i-1$, set
	\begin{equation*}
		z^{i,j}_l\defeq x_{i+1}\cdots x_{j-1-l}y_{j-l}\cdots y_{j-1}
	\end{equation*}
	with $z_l^{i,i+1}=1$.
\end{notation}

These should be thought of as truncated versions of the $z_l$, with all indices lying strictly between $i$ and $j$, but still $l$ terms involving ``$y$'' appearing, so $z_l=z_l^{1,n}$.

For example, keeping $n=5$ we have
\begin{align*}
	z_0^{1,4}&=x_2x_3\\
	z_1^{1,4}&=x_2y_3\\
	z_2^{1,4}&=y_2y_3
\end{align*}

\begin{proposition}\label{zProp}
	For any $1\leq i<j\leq n$ and $0\leq l\leq j-i-1$, we have $z^{i,j}_l\delta_{i,j}\in\mathfrak{a}$.
\end{proposition}

\begin{proof}
	The proof will be by induction on $j-i$.

	If $j-i=1,$ then $\delta_{i,j}=\delta_{i,i+1}\in\mathfrak{a}$ since $\mathfrak{a}=\mathcal{J}(P_n)$, and so $z\delta_{i,j}$ belongs to $\mathfrak{a}$ for any $z\in R$.

	Now suppose that $j-i\geq 2$ and that the result holds for all smaller such differences.

	If $l=0$ then $x_{i+1}\mid z_l^{i,j}$, and so
	\begin{align*}
		z_l^{i,j}\delta_{i,j}&=z_l^{i,j}(x_iy_j-x_jy_i)\\
		&=z_l^{i,j}x_iy_j-z_l^{i+1,j}x_ix_jy_{i+1}+z_l^{i+1,j}x_ix_jy_{i+1}-z_l^{i,j}x_jy_i\\
		&=x_iz_l^{i+1,j}(x_{i+1}y_j-x_jy_{i+1})+x_jz_l^{i+1,j}(x_iy_{i+1}-x_{i+1}y_i)\\
		&=x_i(z_l^{i+1,j}\delta_{i+1,j})+(x_jz_l^{i+1,j})\delta_{i,i+1}
	\end{align*}
	We have $z_l^{i+1,j}\delta_{i+1,j}\in\mathfrak{a}$ by the inductive hypothesis, and $\delta_{i,i+1}\in\mathfrak{a}$ since $\mathfrak{a}=\mathcal{J}(P_n)$, so $z_l^{i,j}\delta_{i,j}\in\mathfrak{a}$.

	Similarly, if $l\neq0$ then $y_{j-1}\mid z_l^{i,j}$, and so
	\begin{align*}
		z_l^{i,j}\delta_{i,j}&=z_l^{i,j}(x_iy_j-x_jy_i)\\
		&=z_l^{i,j}x_iy_j-z_{l-1}^{i,j-1}x_{j-1}y_iy_j+z_{l-1}^{i,j-1}x_{j-1}y_iy_j-z_l^{i,j}x_jy_i\\
		&=y_jz_{l-1}^{i,j-1}(x_iy_{j-1}-x_{j-1}y_i)+y_iz_{l-1}^{i,j-1}(x_{j-1}y_j-x_jy_{j-1})\\
		&=y_j(z_{l-1}^{i,j-1}\delta_{i,j-1})+(y_iz_{l-1}^{i,j-1})\delta_{j-1,j}
	\end{align*}
	We have $z_{l-1}^{i,j-1}\delta_{i,j-1}\in\mathfrak{a}$ by the inductive hypothesis, and $\delta_{j-1,j}\in\mathfrak{a}$ since $\mathfrak{a}=\mathcal{J}(P_n)$, so $z_l^{i,j}\delta_{i,j}\in\mathfrak{a}$.

	Then, in either case, we have $z_l^{i,j}\delta_{i,j}\in\mathfrak{a}$ as desired.
\end{proof}

\begin{corollary}
	We have $\mathfrak{b}\subseteq(\mathfrak{a}:\mathfrak{g})$.
\end{corollary}

\begin{proof}
	This follows immediately from \cref{zProp}, since for any $1\leq i<j\leq n$ and $0\leq l\leq n-2$, we have that $z_{l'}^{i,j}\mid z_l$ for some $0\leq l'\leq n-2$, $l'$ is simply the number of occurrences of ``$y$''s in $z_l$ with indices strictly between $i$ and $j$.
\end{proof}

We will now show the reverse inclusion, for which we need a few preparatory results:

\begin{lemma}\label{ColonLemma}
	For any $2\leq i\leq n-1$, we have
	\begin{equation*}
		(\mathfrak{a}:\delta_{i-1,i+1})=(x_i,y_i)+\mathfrak{a}
	\end{equation*}
\end{lemma}

\begin{proof}
	We have
	\begin{align*}
		(\mathfrak{a}:\delta_{i-1,i+1})&=\bigcap_{\mathclap{S\in\mathcal{C}(P_n)}}\,(P_S(P_n):\delta_{i-1,i+1})=\bigcap_{\mathclap{\substack{S\in\mathcal{C}(P_n)\\\delta_{i-1,i+1}\notin P_S(P_n)}}}\,P_S(P_n)=\bigcap_{\mathclap{\substack{S\in\mathcal{C}(P_n)\\i\in S\\i-1,i+1\notin S}}}\,P_S(P_n)\\
		&=(x_i,y_i)+\mathcal{J}(P_{\{1,\ldots,i-1\}})+\mathcal{J}(P_{\{i+1,\ldots,n\}})=(x_i,y_i)+\mathfrak{a}
	\end{align*}
	since the $P_S(P_n)$ are prime.

	Here $P_{\{v_1,\ldots,v_t\}}$ denotes the graph on $v_1,\ldots,v_t$ with edges $\{v_1,v_2\},\ldots,\{v_{t-1},v_t\}$, that is, the path on $v_1,\ldots,v_t$.
\end{proof}

\pagebreak

\begin{proposition}\label{GradingProp}
	Every homogeneous element of $R$ with degree $(0,1,\ldots,1,0)$ belongs to $\mathfrak{b}$.
\end{proposition}

\begin{proof}
	To see this, simply note that $x_iy_{i+1}=x_{i+1}y_i$ modulo $\mathfrak{a}$ for every $1\leq i\leq n-1$, and so we may ``swap'' adjacent terms of ``$x$''s and ``$y$''s in our $z_l$ whilst staying within $\mathfrak{b}$.
\end{proof}

\begin{corollary}
	We have $(\mathfrak{a}:\mathfrak{g})\subseteq\mathfrak{b}$.
\end{corollary}

\begin{proof}
	Take any $c\in(\mathfrak{a}:\mathfrak{g})$. Since $\mathfrak{a}$ and $\mathfrak{g}$ are homogeneous we have that $(\mathfrak{a}:\mathfrak{g})$ is homogeneous also, and so we may assume that $c$ is homogeneous. Because $c\in(\mathfrak{a}:\mathfrak{g})$, we have in particular that $c\delta_{i-1,i+1}\in\mathfrak{a}$ for every $2\leq i\leq n-1$ since $\delta_{i-1,i+1}\in\mathfrak{g}$.

	In \cref{ColonLemma} we saw that
	\begin{equation*}
		(\mathfrak{a}:\delta_{i-1,i+1})=(x_i,y_i)+\mathfrak{a}
	\end{equation*}
	and so we can write $c=f_2+a_2$ for some $f_2\in(x_2,y_2)$ and $a_2\in\mathfrak{a}$. Note that $f_2=c-a_2\in(\mathfrak{a}:\mathfrak{g})$ also, so we can write $f_2=f_3+a_3$ for some $f_3\in(x_3,y_3)$ and $a_3\in\mathfrak{a}$. Continuing in this way, we arrive at
	\begin{equation*}
 	   c=f_{n-1}+(a_2+\cdots+a_{n-1})
	\end{equation*}
	where either $x_i\mid f_{n-1}$ or $y_i\mid f_{n-1}$ for each $2\leq i\leq n-1$. This means that each term in $f_{n-1}$ is divisible by a homogeneous element of $R$ of degree $(0,1,\ldots,1,0)$, so $f_{n-1}\in\mathfrak{b}$ by \cref{GradingProp}. Then $c\in\mathfrak{b}$ also, and we are done.
\end{proof}

This concludes the proof of \cref{ColonTheorem} for $G=K_n$. We will now extend this result.

\begin{tcolorbox}
	We now allow $G$ to be any Hamiltonian graph on $n$ vertices, otherwise our notation is as before. Since $G$ is Hamiltonian it contains $C_n$, and we may assume that
	\vspace{-0.5\baselineskip}
	\begin{equation*}
		E(C_n)=\{\{1,2\},\ldots,\{n-1,n\},\{1,n\}\}
	\end{equation*}
\end{tcolorbox}

Since $K_n$ is Hamiltonian, to prove \cref{ColonTheorem} for all Hamiltonian graphs, it will suffice to show that $(\mathfrak{a}:\mathfrak{g})$ is independent of the choice of Hamiltonian graph:

\begin{proposition}\label{HamiltonianProp}
	We have $(\mathfrak{a}:\mathfrak{g})=(\mathfrak{a}:\delta_{1,n})$.
\end{proposition}

\begin{proof}	
	We will first show that $(\mathfrak{a}:\delta_{1,n})\subseteq(\mathfrak{a}:\delta_{i,j})$ for any $1\leq i<j\leq n$.

	We know
	\begin{equation*}
		(\mathfrak{a}:\delta_{i,j})=\myCap_{S\in\mathcal{C}(P_n)}(P_S(P_n):\delta_{i,j})=\myCap_{\substack{S\in\mathcal{C}(P_n)\\\delta_{i,j}\notin P_S(P_n)}}P_S(P_n)
	\end{equation*}
	since the $P_S(P_n)$ are prime, and $\delta_{1,n}\in P_S(P_n)$ if and only if $S=\varnothing$. Clearly $\delta_{i,j}\in P_\varnothing(P_n)$ for any $1\leq i<j\leq n$, and so the desired containment holds.

	Then we have
	\begin{equation*}
		(\mathfrak{a}:\mathfrak{g})=\mathfrak{a}:\hspace{-0.1cm}\left(\hspace{0.5cm}\mySum_{\{i,j\}\in E(G)}(\delta_{i,j})\hspace{-0.05cm}\right)=\;\myCap_{\{i,j\}\in E(G)}(\mathfrak{a}:\delta_{i,j})=(\mathfrak{a}:\delta_{1,n})
	\end{equation*}
	with the last equality following since $G$ contains $C_n$ because it is Hamiltonian and so $\{1,n\}\in E(G)$.
\end{proof}

This concludes the proof of \cref{ColonTheorem} for all Hamiltonian graphs.\qed

As an example application of \cref{ColonTheorem}, we show the following:

\begin{proposition}
	We have
	\begin{equation*}
		\Ass_R(H_\mathfrak{g}^{n-1}(R))=\{\mathcal{J}(K_n)\}
	\end{equation*}
\end{proposition}

\begin{proof}
	By \cref{ColonTheorem} and \cref{HamiltonianProp}, we have that $\mathfrak{b}/\mathfrak{a}$ is a root of $H_\mathfrak{g}^{n-1}(R)$, and so
	\begin{equation*}
		\Ass_R(H_\mathfrak{g}^{n-1}(R))=\Ass_R(\mathfrak{b}/\mathfrak{a})
	\end{equation*}
	by \cref{RootAssEqual}.

	Let $\mathfrak{p}$ be any associated prime of $\mathfrak{b}/\mathfrak{a}$, so $\mathfrak{p}=\Ann_R(b+\mathfrak{a})$ for some non-zero $b+\mathfrak{a}\in\mathfrak{b}/\mathfrak{a}$. This is equivalent to saying that $\mathfrak{p}=(\mathfrak{a}:b)$.

	Since $b\in\mathfrak{b}=(\mathfrak{a}:\mathcal{J}(K_n))$ we have $b\mathcal{J}(K_n)\subseteq\mathfrak{a}$, so certainly $\mathcal{J}(K_n)\subseteq\mathfrak{p}$. We will now show the converse.

	By definition, we have $\mathfrak{p}b\subseteq\mathfrak{a}\subseteq\mathcal{J}(K_n)$. Now, $\mathcal{J}(K_n)=P_\varnothing(K_n)$ is prime, and so either $\mathfrak{p}\subseteq\mathcal{J}(K_n)$ as desired, or $b\in\mathcal{J}(K_n)$. We know that $b\mathcal{J}(K_n)\subseteq\mathfrak{a}$, and so in this case we have $b^2\in\mathfrak{a}$. But $\mathfrak{a}=\mathcal{J}(P_n)$ is radical, so $b\in\mathfrak{a}$. Then $b+\mathfrak{a}=0+\mathfrak{a}$, which contradicts that $b+\mathfrak{a}$ is non-zero, and so we are done.
\end{proof}

\section*{Acknowledgements}

Thanks to Dr Moty Katzman for suggesting the use of the Nagel-Schenzel Isomorphism to study the local cohomology modules of binomial edge ideals, and for many helpful discussions. Thanks also to Dr Josep \`{A}lvarez Montaner and Dr Paul Johnson for their helpful feedback and suggestions regarding these results. 
\vfill
\nocite{M2}
\printbibliography

\end{document}